\documentclass[12pt]{article}

\usepackage[utf8]{inputenc}
\usepackage[T1]{fontenc}
\usepackage{amsmath, amssymb, amsthm}
\usepackage{mathtools}
\usepackage{geometry}
\usepackage{booktabs}
\usepackage{array}
\usepackage{graphicx}
\usepackage{xcolor}
\usepackage{hyperref}
\usepackage{enumerate}
\usepackage{tikz}

\usetikzlibrary{calc}

\definecolor{myteal}{HTML}{9FE1CB}
\definecolor{mytealline}{HTML}{0F6E56}
\definecolor{mycoral}{HTML}{F0997B}
\definecolor{mycoralline}{HTML}{993C1D}
\newtheorem{theorem}{Theorem}
\newtheorem{definition}[theorem]{Definition}
\newtheorem{corollary}[theorem]{Corollary}
\newtheorem{lemma}[theorem]{Lemma}

\newtheorem{postulate}[theorem]{Postulate}

\newtheorem{remark}[theorem]{Remark}

\newcommand{\Da}{D^{(\alpha)}}
\newcommand{\Xa}{X^{(\alpha)}}
\newcommand{\Lapa}{\Delta^{(\alpha)}}
\newcommand{\grada}{\operatorname{grad}^{(\alpha)}}
\newcommand{\diva}{\operatorname{div}^{(\alpha)}}
\newcommand{\curla}{\operatorname{curl}^{(\alpha)}}
\newcommand{\dbara}{\overline\partial^{(\alpha)}}
\newcommand{\dela}{\partial^{(\alpha)}}
\newcommand{\mua}{\mu^{(\alpha)}}
\newcommand{\laa}{\lambda^{(\alpha)}}
\newcommand{\nua}{\nu^{(\alpha)}}
\newcommand{\kaa}{\kappa^{(\alpha)}}
\newcommand{\caa}{c^{(\alpha)}}

\title{%
  \textbf{Quaternionic Elasticity Theory with Local Fractional Derivatives}
}
\small{
\author
{Hung Manh Nguyen$^{(1)\footnote{corresponding author}}$, Juan Bory-Reyes$^{(2)}$ }
\vskip 1truecm
\date{\small $^{(1)}$ {University of Transport and Communications, Hanoi, Vietnam}\\Email: hung.manh.nguyen@utc.edu.vn\\$^{(2)}$ {ESIME-Zacatenco-Instituto Polit\'ecnico Nacional. Ciudad de M\'exico, M\'exico}\\Email: juanboryreyes@yahoo.com }}

\begin{document}
\maketitle

\begin{abstract}
\noindent 
{We develop a quaternionic formulation of linear elasticity in which the classical derivative is replaced by Yang's local fractional derivative and the classical Cauchy--Riemann operator by a fractal Cauchy--Riemann operator. We formulate fractal kinematics, a constitutive law with a material length scale and the equilibrium equation, and show that, for displacement fields with polynomial components, the resulting homogeneous fractal Lam\'e system is equivalent to the monogenicity of an explicit quaternionic field. On the polynomial sector, that is, for finite combinations of normalized fractal monomials, we prove a fractal Kolosov--Muskhelishvili representation of the displacement field in terms of two monogenic quaternionic potentials, together with its completeness, by a direct construction that uses neither the Galerkin nor the Papkovich--Neuber potentials. The position term of the classical formula has to be replaced by a corrected position operator, which satisfies the exact Weyl relation with the partial local fractional derivatives. Finally a fractal Papkovich-Neuber representation is derived.}
\end{abstract}


\section{Introduction}
\label{sec:intro}

Holomorphic function theory has long been an effective tool for
linear elasticity in the plane: the displacement and stress fields
are represented in terms of a holomorphic and an anti-holomorphic
function, the classical Kolosov--Muskhelishvili formulae
\cite{Muskhelishvili1977}. In $\mathbb R^3$, quaternion analysis
provides the natural generalization, with monogenic functions (null solutions
of a generalized Cauchy--Riemann operator) taking over the
role of holomorphic ones. Bock and G\"urlebeck \cite{BockGurlebeck2009}
obtained a first generalized spatial Kolosov--Muskhelishvili formula this way, via a quaternionic Goursat-type representation of biharmonic functions;
Weisz-Patrault, Bock and G\"urlebeck \cite{WeiszPatraultBockGurlebeck2014}
and Bock \cite{Bock2014} developed the corresponding Papkovich--Neuber
theory and monogenic series expansions, building on Klein Obbink's
decomposition of {harmonic} functions \cite{KleinObbink1993}. A
recurring difficulty in this line of work is redundancy: quaternion-valued
monogenic functions live in $\mathbb H\cong\mathbb R^4$, one dimension
too many for a displacement field valued in $\mathbb R^3$, so a
nontrivial kernel has to be removed by hand when assembling a basis of
solutions. \'Alvarez-Pe\~na and Porter \cite{AlvarezPena2013,AlvarezPenaPorter2014}
{studied the underlying structure: a harmonic function with values in {the set of reduced quaternions} can be decomposed into a monogenic and an anti-monogenic function, but in general only with quaternion-valued components, so that the decomposition leaves the paravector space \cite{Bock2014}; the additional component is the source of the redundancy.} G\"urlebeck and Nguyen
\cite{GurlebeckNguyen2013,GurlebeckNguyen2014} and G\"urlebeck,
Nguyen and Legatiuk \cite{GurlebeckNguyenLegatiuk2014} {addressed this issue}
by introducing $\psi$-hyperholomorphic functions{, i.e., null solutions of
a generalized Cauchy-Riemann operator built from an arbitrary
structural set $\psi$} and showing that every paravector-valued harmonic
function decomposes into a monogenic, an
anti-monogenic, and a $\psi$-hyperholomorphic part. Bock, G\"urlebeck,
Legatiuk and Nguyen \cite{BockLegatiukNguyen2015} then used exactly
this three-way decomposition to build an alternative
Kolosov--Muskhelishvili formula for the displacement field that stays
entirely within $\mathbb R^3$, with a basis of solutions free of the
earlier redundancy; the same circle of ideas has since been coupled to
the finite element method for crack-propagation problems
\cite{Legatiuk2021}. It is this {powerful mathematical method}, and in particular the Kolosov--Muskhelishvili representation built from quaternionic monogenic functions, that the present paper sets out to {transfer} to a fractal setting.

A second, independent line of development makes that transplant
possible. Real materials are frequently non-smooth at the scale that
matters: porous rock, fractured concrete, biological tissue, and
many manufactured composites all show self-similar or fractal
microstructure over some range of scales, and classical continuum
mechanics {building on ordinary differentiability} does not see
this structure directly. Yang's \emph{local fractional derivative}
\cite{Yang2015,Yang2012book} was introduced
precisely to differentiate non-differentiable, fractal-supported
functions on Cantor-type sets, and has since been used to reformulate
a wide range of partial differential equations {such as} heat conduction,
wave propagation, diffusion directly on such sets \cite{RahmatBaleanuYang2015}, 
with a local fractional vector calculus (gradient, divergence, curl) developed more recently
to support field-theoretic applications \cite{VecCalc2025}. Within
this circle, Bory-Reyes and P\'erez-de la Rosa
\cite{BoryRosas2021} combined Yang's local fractional derivative with
quaternion (Moisil--Teodorescu) analysis, {defining a fractal Moisil--Teodorescu operator on 
Cantor-type cylindrical and spherical coordinate systems, factorizing the fractal Laplacian and 
applying it to the local fractional Helmholtz equation.}

Fractional and fractal generalizations of elasticity theory itself
have been pursued from several directions, largely independent of the
quaternionic viewpoint above: Carpinteri and Cornetti
\cite{CarpinteriCornetti2002} used fractional calculus to describe
stress and strain localization in fractal media; Sumelka and
B\l aszczyk \cite{SumelkaBlaszczyk2014} and Lazopoulos
\cite{Lazopoulos2019} developed fractional-strain continua for linear
elasticity; Balankin \cite{Balankin2015} proposed a continuum
framework for fractal materials based on a fractal metric; and
Patnaik, Sidhardh and Semperlotti \cite{PatnaikSidhardhSemperlotti2020,PatnaikSemperlotti2020}
built fractional-order plate and elastodynamic models for nonlocal,
attenuating media. These theories are motivated by the same physical
observation as above{: damage, porosity, and roughness at
small scales leave a fractal signature that a classical, smooth
continuum model cannot capture} with applications ranging from
rock and soil mechanics to the crack-propagation and damage problems
already touched on by the function-theoretic methods of
\cite{Legatiuk2021}. {To the best of our knowledge, none of these fractional-elasticity theories is built on the quaternionic representation-formula machinery of the first paragraph above. In the quaternionic direction, fractional {(or metric)} deformations of the Moisil--Teodorescu operator have been introduced, in particular by means of Stillinger's formalism \cite{BoryPerezPena2022}, {by metric differential calculus in a three-dimensional continum with a non-Euclidean metric \cite{BalankinBoryShapiro2016}} and of derivatives in non-integer dimensional space \cite{BoryNIDS2026}, the latter also for the quaternionic Lam\'e--Navier operator{. A solution-representation theory of Kolosov--Muskhelishvili type for fractal media does not appear to have been developed.}}

The present paper combines these two lines of work. {Using Yang's local fractional derivative and the fractal Cauchy--Riemann operator described above, we formulate a quaternionic theory of linear elasticity with local fractional derivatives, and prove a fractal analogue of the Kolosov--Muskhelishvili representation of the classical theory surveyed above. Our aim is to show that this representation, together with its completeness, survives fractalization when restricted to the polynomial sector, provided that the position term of the classical formula is replaced by a suitably corrected operator.}

{This work proceeds as follows}. Section~\ref{sec:prelim} develops
the fractal quaternionic calculus needed throughout: Yang's local
fractional derivative, the quaternion algebra and the fractal
Cauchy--Riemann operator. Section~\ref{sec:frac_elas} introduces fractal kinematics, the compatibility conditions, the constitutive law, and
the equilibrium equation as structured postulates and derives the fractal Lam\'e equation {together with its quaternionic factorization (Theorem~\ref{thm:Q}).} Section~\ref{sec:sol_rep} {deals with} the main results. {More in detail, we will first construct} a corrected {position operator} and establishes an exact Weyl relation, and then proves the fractal Kolosov--Muskhelishvili representation of the polynomial solutions in both directions by a direct construction. {An illustrative example is presented in Section~\ref{sec:examples}. Finally, conclusions of this work and further outlook are given in Section~\ref{sec:concl}}.

\section{Preliminaries}
\label{sec:prelim}

\subsection{Classical Elasticity}

For completeness and to fix notation, we briefly recall the classical linear elasticity equations in the smooth setting. Let $\mathbf u$ be the displacement field. {Next we define the infinitesimal strain tensor, which is defined as} 
\begin{equation}\label{eq:strain_classical}
  \varepsilon_{ij} = \frac12\left(\frac{\partial u_i}{\partial x_j}+\frac{\partial u_j}{\partial x_i}\right),
  \qquad
  \boldsymbol\varepsilon = \frac12\bigl(\nabla\mathbf u+(\nabla\mathbf u)^T\bigr).
\end{equation}
For a homogeneous, isotropic material, Hooke's law reads
\begin{equation}\label{eq:hooke_classical}
  \sigma_{ij}=\lambda\,\delta_{ij}\,\varepsilon_{kk}+2\mu\,\varepsilon_{ij}.
\end{equation}
Static equilibrium in the absence of inertia is given by
\begin{equation}\label{eq:equilibrium_classical}
  \sum_{j=1}^{3}\frac{\partial\sigma_{ij}}{\partial x_j}+f_i=0,
  \quad i=1,2,3,
  \qquad\text{i.e.}\qquad
  \nabla\cdot\boldsymbol\sigma + \mathbf f = \mathbf 0.
\end{equation}
Combining \eqref{eq:strain_classical}--\eqref{eq:equilibrium_classical} yields the Lam\'e--Navier equation
\begin{equation}\label{eq:lame_classical}
  \mu\,\Delta\mathbf u+(\lambda+\mu)\,\nabla(\nabla\cdot\mathbf u)+\mathbf f=\mathbf 0,
\end{equation}
{where $\mu > 0$, $\lambda > -\frac{2}{3}\mu$ are the so-called Lamé coefficients}.

\subsection{Fractal Quaternionic Calculus}
\label{subsec:toolkit}

{In this section, we go over some basic facts of fractal vector-calculus in Cartesian coordinates $x_0,x_1,x_2$ that will be useful in this work}: Yang's local fractional
derivative, the quaternion algebra, and a fractal Cauchy--Riemann operator {that is obtained by replacing classical partial derivatives with }Yang's local fractional derivative~\cite{Yang2012book}.

\begin{definition}[Local fractional derivative]
\label{def:D_alpha}
A function $f:\mathbb R\to\mathbb R$ is called \emph{local fractional continuous} at $x_0$ (in the sense of Yang) if for every $\varepsilon>0$ there exists $\delta>0$ such that
$$
|f(x)-f(x_0)| < \varepsilon^\alpha \quad\text{whenever}\quad |x-x_0|<\delta,
$$
with $0<\alpha\le 1$. We write $f\in C^\alpha(a,b)$ if $f$ is local fractional continuous at every point of $(a,b)$.

For $f\in C^\alpha(a,b)$, the local fractional derivative of order $\alpha\in(0,1]$ at $x_0\in(a,b)$ is defined by
$$
D^{(\alpha)}f(x_0)
:= \lim_{x\to x_0}\Gamma(1+\alpha)\,\frac{f(x)-f(x_0)}{(x-x_0)^\alpha}.
$$
For a function of several variables, partial local fractional derivatives are defined by applying this one-variable limit to each coordinate while keeping the others fixed.
\end{definition}

This fractional derivative is $\mathbb R$-linear: {for $c_1,c_2\in\mathbb R$ and functions $f,g$ that are local fractional differentiable on $(a,b)$,}
\begin{equation}
  D^{(\alpha)}[c_1f+c_2g] = c_1\,D^{(\alpha)}f+c_2\,D^{(\alpha)}g.
  \tag{L}\label{rule:lin}
\end{equation}
The generalized Mittag-Leffler, sine, cosine functions on a fractal set are {defined by}
$$
  \begin{gathered}
  E_\alpha(x^\alpha):=\sum_{k\ge0}\frac{x^{k\alpha}}{\Gamma(1+k\alpha)},
  \qquad
  \sin_\alpha(x^\alpha):=\sum_{k\ge0}\frac{(-1)^kx^{(2k+1)\alpha}}{\Gamma(1+(2k+1)\alpha)},\\[2pt]
  \cos_\alpha(x^\alpha):=\sum_{k\ge0}\frac{(-1)^kx^{2k\alpha}}{\Gamma(1+2k\alpha)}.
  \end{gathered}
$$
They satisfy
\begin{equation}
  D^{(\alpha)}\Big[m_k(x)\Big] = m_{k-1}(x),\quad\text{with}\quad m_k(x) = \dfrac{x^{k\alpha}}{\Gamma(1+k\alpha)}
  \tag{M}\label{rule:monomial}
\end{equation}
(interpreted as $0$ for $k=0$), and
\begin{equation}
  D^{(\alpha)}E_\alpha(x^\alpha) = E_\alpha(x^\alpha),
  \hspace{0.5cm}
  D^{(\alpha)}\sin_\alpha(x^\alpha) = \cos_\alpha(x^\alpha),
  \hspace{0.5cm}
  D^{(\alpha)}\cos_\alpha(x^\alpha) = -\sin_\alpha(x^\alpha).
  \tag{T}\label{rule:trig}
\end{equation}
{Let $C^\alpha\subset\mathbb R$ a Cantor-type subset of Hausdorff dimension $\alpha\in(0,1)$} (for the standard middle-third Cantor set, $\alpha=\ln2/\ln3\approx0.6309$), and consider functions of three variables $(x_0,x_1,x_2)$ with each $x_i\in C^\alpha$. The set $C^\alpha$ is equipped with the natural fractal measure $d^{(\alpha)}x=(dx)^\alpha/\Gamma(1+\alpha)$, and we denote by $D_i^{(\alpha)}$ the partial local fractional derivative with respect to $x_i$, i.e., $D^{(\alpha)}$ applied to $x_i$ with the other two coordinates fixed ($i=0,1,2$).

{Let $\mathbb H$ denote the non-commutative algebra of real quaternions}. An element $q\in\mathbb H$ is written as
\[
  q = \sum_{k=0}^3 q_k \mathbf e_k = q_0 + q_1\mathbf e_1 + q_2\mathbf e_2 + q_3\mathbf e_3,
  \qquad q_k\in\mathbb R,
\]
with $\mathbf e_0=1$ and imaginary units satisfying $\mathbf e_1^2=\mathbf e_2^2=\mathbf e_3^2=-1$ and
$\mathbf e_1\mathbf e_2=\mathbf e_3$, $\mathbf e_2\mathbf e_3=\mathbf e_1$, $\mathbf e_3\mathbf e_1=\mathbf e_2$.
Writing $q=q_0+\vec q$ with scalar part $\mathrm{Sc}(q)=q_0\in\mathbb R$ and vector part
$\mathrm{Vec}(q)=\vec q=q_1\mathbf e_1+q_2\mathbf e_2+q_3\mathbf e_3\in\mathbb R^3$,
the product of two quaternions $p,q$ takes the form
\[
  pq = p_0q_0 - \langle\vec p,\vec q\rangle + p_0\vec q + q_0\vec p + [\vec p,\vec q],
\]
where $\langle\cdot,\cdot\rangle$ is the Euclidean inner product in $\mathbb R^3$ and
$[\vec p,\vec q]=\vec p\times\vec q$ is the cross product. Thus $\mathbb H\cong\mathbb R\oplus\mathbb R^3$ as a direct sum of scalar and vector parts.

{The reduced quaternion set, denoted by $\mathcal A$, is defined as the subset of $\mathbb H$ generated by the basis elements $(1, e_1,e_2)$ as follows:} 
\[
  \mathcal A := \{q\in\mathbb H : q = q_0 + q_1\mathbf e_1 + q_2\mathbf e_2\}
  \cong \mathbb R\oplus\mathbb R^2,
\]
i.e., quaternions whose vector part lies in the $(\mathbf e_1,\mathbf e_2)$-plane. Note that $\mathcal A$ is only a real-linear subspace of $\mathbb H$, {it is not a subalgebra, since} $\mathbf e_1\mathbf e_2=\mathbf e_3\notin \mathcal A$. The product of two {reduced quaternions need not belong to $\mathcal A$.}

{Throughout the paper, the position variable and the displacement field are viewed as $\mathcal A$-valued via the identification $\mathcal A\cong \mathbb R^3$ and can be represented in terms of the basis as}
\begin{equation}\label{eq:A-embedding}
  x := x_0\mathbf e_0+x_1\mathbf e_1+x_2\mathbf e_2\in \mathcal A,
  \qquad
  \mathbf u := u_0\mathbf e_0+u_1\mathbf e_1+u_2\mathbf e_2\in \mathcal A.
\end{equation}

\begin{definition}[Fractal Cauchy--Riemann operator]
\label{def:DCR_classical}
Let $\Omega\subset\mathbb R^3$ and let $f:\Omega\to\mathbb H$ be a quaternion-valued function,
\[
  f(x) = \sum_{i=0}^3 f_i(x)\,\mathbf e_i,
  \qquad f_i:\Omega\to\mathbb R.
\]
The \emph{fractal Cauchy--Riemann operator} and its conjugate are defined by
\begin{equation}\label{eq:DCR}
  \overline\partial^{(\alpha)} f
  := \mathbf e_0 D_0^{(\alpha)} f + \mathbf e_1 D_1^{(\alpha)} f + \mathbf e_2 D_2^{(\alpha)} f,
  \qquad
  \partial^{(\alpha)} f
  := \mathbf e_0 D_0^{(\alpha)} f - \mathbf e_1 D_1^{(\alpha)} f - \mathbf e_2 D_2^{(\alpha)} f,
\end{equation}
where the partial local fractional derivatives act componentwise:
\[
  D_k^{(\alpha)}f := \sum_{i=0}^3 \bigl(D_k^{(\alpha)}f_i\bigr)\,\mathbf e_i,
  \qquad k=0,1,2.
\]
A function $f$ is {said to be} \emph{(left) fractal-monogenic} in $\Omega$ if
$\overline\partial^{(\alpha)} f = 0$ in $\Omega$.
\end{definition}

At $\alpha=1$, we have $D_i^{(1)}=\partial_{x_i}$ and
\[
  \overline\partial^{(1)} f
  = \mathbf e_0\,\partial_{x_0}f + \mathbf e_1\,\partial_{x_1}f + \mathbf e_2\,\partial_{x_2}f,
\]
which is the classical Cauchy--Riemann operator of quaternionic analysis in $\mathbb R^3$.

The fractal Laplacian on quaternion-valued {vector} fields is defined componentwise by
\[
  \Delta^{(\alpha)} f
  := \sum_{i=0}^3 \bigl(\Delta^{(\alpha)} f_i\bigr)\,\mathbf e_i,
  \quad\text{where}\quad
  \Delta^{(\alpha)} f_i
  := \sum_{k=0}^2 D_k^{(\alpha)} D_k^{(\alpha)} f_i.
\]
As in the classical case, mixed local fractional derivatives commute under the usual regularity assumptions: if $D_i^{(\alpha)}D_j^{(\alpha)}f$ and $D_j^{(\alpha)}D_i^{(\alpha)}f$ exist and are local fractional continuous in a neighborhood of $x$, then
\[
  D_i^{(\alpha)}D_j^{(\alpha)}f(x)=D_j^{(\alpha)}D_i^{(\alpha)}f(x),
\]
for $i\neq j\in\{0,1,2\}$. This is the local fractional analogue of Clairaut’s theorem; see \cite{Yang2012book} for the underlying definitions and related results. Consequently, the fractal Cauchy--Riemann operator factorizes the Laplacian:
\[
  \overline\partial^{(\alpha)}\partial^{(\alpha)} f
  = \partial^{(\alpha)}\overline\partial^{(\alpha)} f
  = \Delta^{(\alpha)} f
\]
for sufficiently regular {vector} fields $f$.

We denote by $\mathcal P^{(\alpha)}$ the $\mathbb R$-vector space of finite linear combinations of monomials
\[
  m_{p_0}(x_0)\,m_{p_1}(x_1)\,m_{p_2}(x_2),
  \qquad p_0,p_1,p_2\in\mathbb Z_{\ge 0},
\]
with real coefficients; we refer to its elements as \emph{fractal solid polynomials}.

For each $n\ge 0$, let $\mathcal P_n^{(\alpha)}\subset\mathcal P^{(\alpha)}$ be the subspace spanned by those monomials with $p_0+p_1+p_2=n$; we call these \emph{fractal-homogeneous of combined degree $n$}. Then
\[
  \mathcal P^{(\alpha)} = \bigoplus_{n\ge 0} \mathcal P_n^{(\alpha)},
\]
and the monomial rule \eqref{rule:monomial} shows that $D_i^{(\alpha)}$ maps $\mathcal P_n^{(\alpha)}$ into $\mathcal P_{n-1}^{(\alpha)}$, preserving the grading. {For a multi-index $p=(p_0,p_1,p_2)\in\mathbb Z_{\ge0}^3$ we write $|p|:=p_0+p_1+p_2$, $p!:=p_0!\,p_1!\,p_2!$ and
\[
  m_p(x):=m_{p_0}(x_0)\,m_{p_1}(x_1)\,m_{p_2}(x_2),
\]
and we denote by $\xi_i\in\mathbb Z_{\ge0}^3$ the $i$-th unit multi-index, $(\xi_i)_j=\delta_{ij}$, so that $p\pm\xi_i$ differs from $p$ only in the $i$-th entry, which is changed by $\pm1$. The monomials $m_p$ form a basis of $\mathcal P^{(\alpha)}$, and $\mathcal P_n^{(\alpha)}$ is spanned by those with $|p|=n$. In this notation, $D_i^{(\alpha)}m_p=m_{p-\xi_i}$, where $m_{p-\xi_i}:=0$ if $p_i=0$; thus $D_i^{(\alpha)}$ is a well-defined $\mathbb R$-linear operator on $\mathcal P^{(\alpha)}$. For $i\neq j$ we have $D_i^{(\alpha)}D_j^{(\alpha)}m_p=D_j^{(\alpha)}D_i^{(\alpha)}m_p=m_{p-\xi_i-\xi_j}$, with the same convention.}

The subspace of fractal-harmonic scalars of combined degree $n$ is
\[
  H_n^{(\alpha)} := \bigl\{ h\in\mathcal P_n^{(\alpha)} : \Delta^{(\alpha)} h = 0 \bigr\}.
\]
Both $\mathcal P_n^{(\alpha)}$ and $H_n^{(\alpha)}$ consist of real-valued functions.

In what follows we will also need vector-valued counterparts of these spaces. We write
$\mathcal P_{n,\mathcal A}^{(\alpha)}$ and $H_{n,\mathcal A}^{(\alpha)}$ for the $\mathcal A$-valued versions
of $\mathcal P_n^{(\alpha)}$ and $H_n^{(\alpha)}$, i.e., fields whose components in the basis
$\{\mathbf e_0,\mathbf e_1,\mathbf e_2\}$ belong to $\mathcal P_n^{(\alpha)}$ and $H_n^{(\alpha)}$, respectively.
Similarly, $\mathcal P_{\mathcal A}^{(\alpha)}$ and $H_{\mathcal A}^{(\alpha)}$ denote the corresponding ungraded spaces.
When convenient, we will also consider $\mathbb H$-valued polynomial fields, defined analogously.

{We will later in Section~\ref{sec:sol_rep} prove that} $H_n^{(\alpha)}$ has the same dimension as the classical space of solid harmonics of degree $n$, namely
\[
  \dim H_n^{(\alpha)} = 2n+1.
\]

For any $h_n^{(\alpha)}\in H_n^{(\alpha)}$, the function
\[
  z_n^{(\alpha)} := \partial^{(\alpha)} h_n^{(\alpha)}
\]
is fractal monogenic, i.e., $\overline\partial^{(\alpha)} z_n^{(\alpha)} = 0$, and reduces at $\alpha=1$ to the classical solid monogenic function associated with $h_n^{(1)}$.

As explicit monogenic polynomial examples in degree $n=1$, consider
\begin{equation*}
  z_{1,1}^{(\alpha)} := m_1(x_1) - m_1(x_0)\,\mathbf e_1,
  \qquad
  z_{1,2}^{(\alpha)} := m_1(x_2) - m_1(x_0)\,\mathbf e_2.
\end{equation*}
Both satisfy $\overline\partial^{(\alpha)} z_{1,k}^{(\alpha)} = 0$ and are the fractal transports of the classical ``complex-variable-like'' generators $z_k = x_k - x_0\mathbf e_k$.

Beyond the polynomial class, one also has explicit non-polynomial harmonic and monogenic functions. For instance,
\[
  h^{(\alpha)} = E_\alpha(x_0^\alpha)\cos_\alpha(x_1^\alpha),
  \qquad
  M^{(\alpha)} = E_\alpha(x_0^\alpha)\cos_\alpha(x_1^\alpha)\,\mathbf e_0
              + E_\alpha(x_0^\alpha)\sin_\alpha(x_1^\alpha)\,\mathbf e_1,
\]
satisfy $\Delta^{(\alpha)}h^{(\alpha)}=0$ and $\overline\partial^{(\alpha)} M^{(\alpha)}=0$ for every $\alpha\in(0,1]$, with $M^{(\alpha)} = \partial^{(\alpha)} h^{(\alpha)}$.

\section{The Fractal Theory of Elasticity}
\label{sec:frac_elas}

\subsection{Fractal Kinematics and Compatibility}

\begin{definition}[Fractal strain]\label{def:strain}
Let $\mathbf u = u_0\mathbf e_0 + u_1\mathbf e_1 + u_2\mathbf e_2$ be a displacement field with scalar components $u_j$. For $i,j\in\{0,1,2\}$, the fractal strain tensor is defined by
\begin{equation}\label{eq:strain_fractal}
  \varepsilon_{ij}^{(\alpha)} := \frac12\bigl(D_i^{(\alpha)}u_j + D_j^{(\alpha)}u_i\bigr).
\end{equation}
\end{definition}

As in the classical theory, these strain components {satisfy} the Saint-Venant compatibility conditions:
\begin{equation}\label{eq:compat_fractal}
  D_k^{(\alpha)}D_l^{(\alpha)}\varepsilon_{ij}^{(\alpha)}
  +D_i^{(\alpha)}D_j^{(\alpha)}\varepsilon_{kl}^{(\alpha)}
  -D_i^{(\alpha)}D_l^{(\alpha)}\varepsilon_{jk}^{(\alpha)}
  -D_j^{(\alpha)}D_k^{(\alpha)}\varepsilon_{il}^{(\alpha)}=0.
\end{equation}
{Indeed, by \eqref{eq:strain_fractal} the left-hand side of \eqref{eq:compat_fractal} is one half of a sum of eight third-order derivatives of the components of $\mathbf u$, which cancel in pairs because mixed local fractional partial derivatives commute}

{Geometrically speaking, \eqref{eq:compat_fractal} rephrases the standard} necessary condition for the locally deformed cells to fit together without gaps or overlaps; we do not use its converse, that is, the existence of a displacement field $\mathbf u$ for a given compatible strain field.

\subsection{Constitutive Law}

The kinematic strain in Definition~\ref{def:strain},
$\varepsilon_{ij}^{(\alpha)}=\tfrac12(D_i^{(\alpha)}u_j+D_j^{(\alpha)}u_i)$,
does not have the classical dimension of strain when $\alpha\neq1$.
Since $[D_i^{(\alpha)}u_j]=[u_j]\,L^{-\alpha}$ and $[u]=L$, one obtains
$$[\varepsilon_{ij}^{(\alpha)}]=L^{1-\alpha}\neq L^0\quad(\alpha\neq1).$$
This is a direct consequence of the definition of the fractal derivatives: applying a derivative of order $\alpha$ to a kinematic quantity generally changes its classical physical dimension.

Lazopoulos~\cite{Lazopoulos2019} {points precisely to the foregoing} for one-dimensional fractional strain and resolves it by passing to a rescaled ``$\Lambda$-fractional'' space. Patnaik and Semperlotti~\cite{PatnaikSemperlotti2020} address the same issue for the fractional deformation-gradient tensor by inserting an explicit length-scale factor $L^{\alpha-1}$ into the derivative itself, noting that these factors ``ensure the dimensional consistency of the deformation gradient tensor.'' In the setting of fractional and fractal continua, this device goes back to Sumelka and B{\l}aszczyk~\cite{SumelkaBlaszczyk2014}; for continua with a genuinely fractal (Hausdorff-dimension) metric, a related resolution-length regularization was introduced by Balankin~\cite{Balankin2015}.

\begin{postulate}[Material length scale; dimensionless constitutive strain]\label{post:length-scale}
Fix a material length scale $\ell>0$, intrinsic to the fractal medium. Define the \emph{constitutive strain}
\begin{equation}\label{eq:strain_hat}
  \widehat\varepsilon_{ij}^{(\alpha)} := \ell^{\,\alpha-1}\,\varepsilon_{ij}^{(\alpha)}.
\end{equation}
\end{postulate}

The parameter $\ell$ represents the lower end of the finite window $(\ell,L)$ over which the material actually behaves in a fractal manner, in the sense of Balankin~\cite{Balankin2015}: below $\ell$ the micro-structure is ordinary discrete matter (grains, fibers), while above $L$ it homogenizes into a classical continuum. The kinematics based on $D^{(\alpha)}$ is intended to model the material only inside this window; $\ell$ pins the otherwise scale-free quantity $\varepsilon_{ij}^{(\alpha)}$ (of dimension $L^{1-\alpha}$) to a physical length, much as a lattice constant pins a continuum elastic modulus to a physical stress in homogenization from a discrete crystal. Its numerical value is a material property, to be fixed by experiment or by homogenization from a micro-mechanical model of the fractal micro-structure.

This assumption is {purely needed} to fix a dimensional mismatch, rather than to state a physical law. {Indeed,} \eqref{eq:strain_hat} asserts that such an $\ell$ exists and that this specific power of it is the appropriate factor with which to multiply $\varepsilon_{ij}^{(\alpha)}$. Both choices are modeled directly on \cite{PatnaikSemperlotti2020,SumelkaBlaszczyk2014}. As a consequence,
\[
  [\widehat\varepsilon_{ij}^{(\alpha)}] = L^{\alpha-1}\cdot L^{1-\alpha} = L^0 \qquad \text{for every }\alpha,
\]
exactly as in the classical theory. In what follows, {any mention of or reference to} ``strain'', are to be understood as the dimensionless constitutive strain $\widehat\varepsilon_{ij}^{(\alpha)}$, and Postulate~\ref{post:hooke} below is stated in terms of it. The underlying idea $\sigma^* = E\,\varepsilon^*$ with $\varepsilon^* = \mathcal D^\alpha u$, {in the scalar one-dimensional case, already appears} in \cite{CarpinteriCornetti2002}.

\begin{postulate}[Fractal Hooke's law]\label{post:hooke}
For an isotropic fractal elastic medium, the stress and constitutive strain are related by
\begin{equation}\label{eq:hooke_fractal}
  \sigma_{ij}^{(\alpha)} = \lambda^{(\alpha)}\,\delta_{ij}\,\operatorname{tr}\widehat\varepsilon^{(\alpha)} + 2\mu^{(\alpha)}\,\widehat\varepsilon_{ij}^{(\alpha)},
\end{equation}
where $\operatorname{tr}\widehat\varepsilon^{(\alpha)} = \widehat\varepsilon_{00}^{(\alpha)} + \widehat\varepsilon_{11}^{(\alpha)} + \widehat\varepsilon_{22}^{(\alpha)}$. The parameters $\lambda^{(\alpha)},\mu^{(\alpha)}$ satisfy:
\begin{enumerate}[(i)] 
  \item {Positivity:} $\mu^{(\alpha)}>0$ and $3\lambda^{(\alpha)}+2\mu^{(\alpha)}>0$, so that the strain energy
        \[
          W^{(\alpha)} = \tfrac12\,\operatorname{tr}\left(\boldsymbol\sigma^{(\alpha)}\widehat{\boldsymbol\varepsilon}^{(\alpha)}\right) > 0,
        \]
  \item {Classical limit:} $\displaystyle\lim_{\alpha\to1}\lambda^{(\alpha)}=\lambda$ and $\displaystyle\lim_{\alpha\to1}\mu^{(\alpha)}=\mu$.
\end{enumerate}
\end{postulate}

Since $\widehat\varepsilon_{ij}^{(\alpha)}$ is dimensionless, \eqref{eq:hooke_fractal} requires $[\lambda^{(\alpha)}] = [\mu^{(\alpha)}] = [\sigma_{ij}^{(\alpha)}]$. The simplest choice compatible with the classical limit is
\[
  [\lambda^{(\alpha)}] = [\mu^{(\alpha)}] = FL^{-2} \quad \text{for every }\alpha,
\]
so that $[\sigma_{ij}^{(\alpha)}] = FL^{-2}$ and $[W^{(\alpha)}] = FL^{-2}$, matching the classical dimensions of stress and energy density at every $\alpha$. No $\alpha$-dependent rescaling of the moduli is needed for dimensional reasons alone; whether their numerical values vary with $\alpha$ or $\ell$ is a separate modelling question.

\subsection{Equilibrium}

\begin{postulate}[Fractal equilibrium]\label{post:equilibrium}
The stress field satisfies the equilibrium equation
\begin{equation}\label{eq:equilibrium_fractal}
  \ell^{\alpha-1}\sum_{j=0}^2 D_j^{(\alpha)}\sigma_{ij}^{(\alpha)} + f_i^{(\alpha)} = 0 \quad \text{in }\Omega^{(\alpha)},
\end{equation}
for each $i\in\{0,1,2\}$, where $\ell$ is the material length scale introduced in Postulate~\ref{post:length-scale}. 
\end{postulate}

The factor $\ell^{\alpha-1}$ is required by dimensional consistency once $[\sigma_{ij}^{(\alpha)}]=FL^{-2}$ is fixed (Postulate~\ref{post:hooke}). The raw divergence $\sum_{j=0}^2 D_j^{(\alpha)}\sigma_{ij}^{(\alpha)}$ has dimension $FL^{-2-\alpha}$, whereas the body-force density has the classical dimension $FL^{-3}$. Multiplying by $\ell^{\alpha-1}$ removes exactly this mismatch,
\[
  \left[\ell^{\alpha-1}D_j^{(\alpha)}\sigma_{ij}^{(\alpha)}\right] = FL^{-2-\alpha}\cdot L^{\alpha-1} = FL^{-3},
\]
so that $f_i^{(\alpha)}$ retains its ordinary physical dimension and numerical value regardless of $\alpha$. This is the same length-scale device as in Postulate~\ref{post:length-scale}, applied here to the stress-divergence operator. Leaving $D_j^{(\alpha)}$ unregularized would force {to the physical strength itself to be reinterpreted} at every choice of $\alpha$, which {lacks a physical basis}.

\subsection{The Fractal Lam\'e Equation}

Substituting $\sigma_{ij}^{(\alpha)}$ from Postulate~\ref{post:hooke}, together with $\widehat\varepsilon_{ij}^{(\alpha)}=\ell^{\alpha-1}\varepsilon_{ij}^{(\alpha)}$ and $\varepsilon_{ij}^{(\alpha)}=\tfrac12(D_i^{(\alpha)}u_j+D_j^{(\alpha)}u_i)$, into \eqref{eq:equilibrium_fractal} gives, for spatially constant $\lambda^{(\alpha)},\mu^{(\alpha)}$,
\begin{equation}\label{eq:lame_fractal}
  \ell^{2(\alpha-1)}\Bigl[\mu^{(\alpha)}\Delta^{(\alpha)}u_i
  +(\lambda^{(\alpha)}+\mu^{(\alpha)})D_i^{(\alpha)}\Bigl(\sum_{j=0}^2D_j^{(\alpha)}u_j\Bigr)\Bigr]
  +f_i^{(\alpha)}=0,
  \qquad i=0,1,2,
\end{equation}
where $\Delta^{(\alpha)}$ is the fractal Laplacian applied component-wise to $\mathbf u$.

The three scalar equations \eqref{eq:lame_fractal} ($i=0,1,2$) combine into a single quaternion equation for $\mathbf u\in\mathcal A$, using both $\overline\partial^{(\alpha)}$ and its conjugate. Define the fractal dilatation
\begin{equation}\label{eq:dilatation}
  e^{(\alpha)} := \sum_{j=0}^2 D_j^{(\alpha)}u_j = \mathrm{Sc}\bigl(\partial^{(\alpha)}\mathbf u\bigr).
\end{equation}
The terms $\mu^{(\alpha)}\Delta^{(\alpha)}u_i$ ($i=0,1,2$) are exactly the $\mathcal A$-components of $\mu^{(\alpha)}\Delta^{(\alpha)}\mathbf u = \mu^{(\alpha)}\overline\partial^{(\alpha)}\bigl[\partial^{(\alpha)}\mathbf u\bigr]$, and, viewing $e^{(\alpha)}$ as a scalar field, $D_i^{(\alpha)}e^{(\alpha)} = \bigl(\overline\partial^{(\alpha)}[e^{(\alpha)}]\bigr)_i$. Both quaternion expressions have zero $\mathbf e_3$-component whenever $\mathbf u\in\mathcal A$, so no information is lost or added by combining the three scalar equations into {only} one:
\begin{equation}\label{eq:lame_DMT}
  \ell^{2(\alpha-1)}\Bigl[\mu^{(\alpha)}\Delta^{(\alpha)}\mathbf u
  +(\lambda^{(\alpha)}+\mu^{(\alpha)})\,\overline\partial^{(\alpha)}e^{(\alpha)}\Bigr]
  +\mathbf f^{(\alpha)} = \mathbf 0.
\end{equation}
The prefactor $\ell^{2(\alpha-1)}$ reflects the two independent length-scale regularizations: one from the constitutive strain \eqref{eq:strain_hat} and one from the equilibrium equation \eqref{eq:equilibrium_fractal}. It reduces to $1$ when $\alpha=1$ and does not multiply $\mathbf f^{(\alpha)}$, which already has the classical body-force dimension $FL^{-3}$ by construction of \eqref{eq:equilibrium_fractal}.

In the absence of body forces, $\mathbf f^{(\alpha)}=\mathbf 0$ (free vibration or static equilibrium with loading applied only through boundary data). Dividing by the non-zero prefactor $\ell^{2(\alpha-1)}$, equation \eqref{eq:lame_DMT} reduces to
\begin{equation}\label{eq:lame_homogeneous}
  \mu^{(\alpha)}\Delta^{(\alpha)}\mathbf u
  +(\lambda^{(\alpha)}+\mu^{(\alpha)})\,\overline\partial^{(\alpha)}e^{(\alpha)}=\mathbf 0.
\end{equation}

We denote by
\begin{equation}\label{eq:Lame-operator}
  L^{(\alpha)}[\mathbf u]
  :=\mu^{(\alpha)}\Delta^{(\alpha)}\mathbf u
  +(\lambda^{(\alpha)}+\mu^{(\alpha)})\,\overline\partial^{(\alpha)}e^{(\alpha)}
\end{equation}
the \emph{fractal Lam\'e operator}, so that \eqref{eq:lame_homogeneous} reads $L^{(\alpha)}[\mathbf u]=\mathbf 0$. The following theorem expresses this operator through the two first-order operators $\overline\partial^{(\alpha)}$ and $\partial^{(\alpha)}$.

\begin{theorem}[Quaternionic factorization of the fractal Lam\'e operator]\label{thm:Q}
Let $\mathbf u\in\mathcal A$. Then
\begin{equation}\label{eq:Q}
  L^{(\alpha)}[\mathbf u]
  =\overline\partial^{(\alpha)}\Bigl[\mu^{(\alpha)}\,\partial^{(\alpha)}\mathbf u
  +(\lambda^{(\alpha)}+\mu^{(\alpha)})\,\mathrm{Sc}\bigl(\partial^{(\alpha)}\mathbf u\bigr)\Bigr].
\end{equation}
In particular, $L^{(\alpha)}[\mathbf u]=\mathbf 0$ if and only if the associated $\mathbb H$-valued field $F:=\mu^{(\alpha)}\partial^{(\alpha)}\mathbf u+(\lambda^{(\alpha)}+\mu^{(\alpha)})\,\mathrm{Sc}(\partial^{(\alpha)}\mathbf u)$ is fractal-monogenic.
\end{theorem}
\begin{proof}
By the factorization $\overline\partial^{(\alpha)}\partial^{(\alpha)}=\Delta^{(\alpha)}$ (Section~\ref{sec:prelim}), the claim follows.
\end{proof}

\begin{remark}
Writing $\partial^{(\alpha)}\mathbf u=e^{(\alpha)}+\mathrm{Vec}(\partial^{(\alpha)}\mathbf u)$, the field in \eqref{eq:Q} is $F=(\lambda^{(\alpha)}+2\mu^{(\alpha)})\,e^{(\alpha)}+\mu^{(\alpha)}\,\mathrm{Vec}(\partial^{(\alpha)}\mathbf u)$, which has in general a non-zero $\mathbf e_3$-component even though $\mathbf u\in\mathcal A$. At $\alpha=1$ identity \eqref{eq:Q} expresses the classical Lam\'e--Navier operator through the quaternionic operators $\overline\partial$ and $\partial$; that it survives fractalization is a consequence of using only $\mathbb R$-linearity and the factorization $\overline\partial^{(\alpha)}\partial^{(\alpha)}=\Delta^{(\alpha)}$.
\end{remark}

\begin{remark}\label{rem:pn-open-classically}
At $\alpha=1$, \eqref{eq:lame_homogeneous} reduces to the classical homogeneous Navier--Lam\'e equation
\[
  \mu\Delta\mathbf u + (\lambda+\mu)\nabla(\nabla\cdot\mathbf u) = \mathbf 0.
\]
{In the classical theory, various representation formulae expressing the solution of the homogeneous Navier-Lam\'e equation are well-known. One of them, due to Kolosov and Muskhelishvili in the plane and extended to $\mathbb R^3$ in quaternionic form in \cite{Bock2014}, expresses $\mathbf u$ through two monogenic quaternionic potentials $\Phi,\Psi$. Section~\ref{sec:sol_rep} establishes a fractal analogue of this representation on the polynomial sector.}
\end{remark}

\section{Solution Representation Formulae}
\label{sec:sol_rep}

In this section we construct a complete solution representation of the fractal Lam\'e equation \eqref{eq:lame_homogeneous} on the polynomial sector, in the form of a fractal Kolosov--Muskhelishvili formula. The necessary toolkit is developed first, followed by the representation theorem and its consequences.

{The operators which we require are the component-wise fractal gradient and divergence, in addition to $\overline\partial^{(\alpha)},\partial^{(\alpha)}$.}
\begin{equation}\label{eq:grad-div}
  \operatorname{grad}^{(\alpha)}g := \sum_{j=0}^2 \mathbf e_j D_j^{(\alpha)}g,
  \qquad
  \operatorname{div}^{(\alpha)}\mathbf V := \sum_{j=0}^2 D_j^{(\alpha)}V_j,
\end{equation}
for a scalar $g$ and a vector $\mathbf V = V_0\mathbf e_0+V_1\mathbf e_1+V_2\mathbf e_2\in\mathcal A$. These satisfy
\[
  \operatorname{div}^{(\alpha)}\mathbf u = \mathrm{Sc}\bigl(\partial^{(\alpha)}[\mathbf u]\bigr),
  \qquad
  \operatorname{grad}^{(\alpha)}g = \overline\partial^{(\alpha)}[g],
\]
so that $\Delta^{(\alpha)}\mathbf u$ and $\operatorname{grad}^{(\alpha)}\operatorname{div}^{(\alpha)}\mathbf u$ reproduce the fractal Lam\'e equation \eqref{eq:lame_homogeneous}. {Its left-hand side is the fractal Lam\'e operator $L^{(\alpha)}$ of \eqref{eq:Lame-operator}, which can be written as}
\begin{equation}\label{eq:Lalpha-def}
  {L^{(\alpha)}[\mathbf u] = \mua\,\Lapa\mathbf u
  + (\laa+\mua)\,\grada\bigl(\diva\mathbf u\bigr).}
\end{equation}

We set 
\[ 
	\nua := \dfrac{\laa}{2(\laa+\mua)}\ ,\qquad 
	\kaa := 4(1-\nua)\ ,\qquad
	\caa := \dfrac{\laa+\mua}{\mua}.
\]
By Postulate~\ref{post:hooke}(i), $\mua>0$ and $\laa+\mua>\tfrac13\mua>0$; moreover $1-2\nua=\tfrac{\mua}{\laa+\mua}$. Hence $\nua<\tfrac12$ and
\begin{equation}\label{eq:constants}
  \caa=\frac{1}{1-2\nua},\qquad
  \kaa-2=2\bigl(1-2\nua\bigr)>0,\qquad
  \caa\bigl(\kaa-2\bigr)=2 .
\end{equation}
By \eqref{eq:Lalpha-def}, $L^{(\alpha)}[\mathbf u]=\mua\bigl(\Lapa\mathbf u+\caa\,\grada\diva\mathbf u\bigr)$. For a vector $\mathbf V=V_0\mathbf e_0+V_1\mathbf e_1+V_2\mathbf e_2\in\mathcal A$ we also set
\[
  \curla\mathbf V:=(\Da_1V_2-\Da_2V_1)\,\mathbf e_0+(\Da_2V_0-\Da_0V_2)\,\mathbf e_1+(\Da_0V_1-\Da_1V_0)\,\mathbf e_2 .
\]

{Operators and constants carrying the superscript $(\alpha)$ always referred to the fractal setting; the same symbols without the superscript ($\partial$, $\overline\partial$, $\Delta$, $\operatorname{grad}$, $\operatorname{div}$, $\mu$, $\nu$) denote their classical counterparts without extra notification} and occur only where classical statements are recalled. Since the operators $\Da_j$ commute on $\mathcal P^{(\alpha)}$, one has $\Lapa\,\grada=\grada\,\Lapa$, $\diva\,\grada=\Lapa$, and, by componentwise expansion,
\begin{equation}\label{eq:curlcurl}
  \curla\curla\mathbf V=\grada\diva\mathbf V-\Lapa\mathbf V .
\end{equation}

\subsection{Fractal Position Operator and Fischer Duality}
\label{subsec:almansi}

{Now, we will introduce a corrected position operator in the fractal setting defined as follows}.

\begin{definition}[Corrected position operator]\label{def:X-operator}
For a polynomial of the form $m_p(x_i)\,g\in\mathcal P^{(\alpha)}$ with $g$ independent of $x_i$ ($p=0,1,2,\ldots$), define
$$X_i^{(\alpha)}\bigl[m_p(x_i)\,g\bigr]:=(p+1)\,m_{p+1}(x_i)\,g,$$
and extend by linearity to all of $\mathcal P^{(\alpha)}$.
\end{definition}

In terms of the multi-index notation of Section~\ref{sec:prelim}, Definition~\ref{def:X-operator} reads $\Xa_i\,m_p=(p_i+1)\,m_{p+\xi_i}$.

\begin{remark}\label{rem:X-grading}
By construction, $X_i^{(\alpha)}$ maps $\mathcal P_n^{(\alpha)}$ into $\mathcal P_{n+1}^{(\alpha)}$, so it preserves the grading and raises the combined degree by one. In this sense, $X_i^{(\alpha)}$ is best viewed as a formal ladder operator on the graded space spanned by $\{m_p(x_i)\}$---i.e., a basis relabeling---rather than as a pointwise operation on functions of $x_i$.
\end{remark}

\begin{lemma}[Exact Weyl relation]\label{lemma:X-weyl}
For every $f\in\mathcal P^{(\alpha)}$,
\[
  D_i^{(\alpha)}X_i^{(\alpha)}f - X_i^{(\alpha)}D_i^{(\alpha)}f = f,
\]
i.e.\ the classical relation $[\partial_{x_i},x_i]=1$ holds exactly, not merely at $\alpha=1$.
\end{lemma}
\begin{proof}
It suffices to verify the identity on polynomials of the form $m_p(x_i)g$ with $g$ independent of $x_i$. By definition, we have
\[
  D_i^{(\alpha)}X_i^{(\alpha)}[m_p(x_i) g] = D_i^{(\alpha)}[(p+1)m_{p+1}(x_i) g] = (p+1)m_{p}(x_i) g,
\]
and for $p\ge 1$,
\[
  X_i^{(\alpha)}D_i^{(\alpha)}[m_p(x_i) g] = X_i^{(\alpha)}[m_{p-1}(x_i) g] = p m_p(x_i) g,
\]
so the commutator equals $(p+1)m_p g - p m_p g = m_p g$. For $p=0$ we have $D_i^{(\alpha)}[g]=0$, hence the commutator is $1\cdot g - 0 = g$. In both cases the result is $f$.
\end{proof}

For $j\neq i$ the operators $\Da_j$ and $\Xa_i$ act on different factors of a basis monomial and therefore commute. Together with Lemma~\ref{lemma:X-weyl} this gives $\Da_j\Xa_i-\Xa_i\Da_j=\delta_{ij}$ on $\mathcal P^{(\alpha)}$, and the operators $\Xa_i$ commute with each other.

For $\varphi\in\mathcal P^{(\alpha)}$ we set $E^{(\alpha)}[\varphi]:=\sum_{i=0}^2\Xa_i\bigl[\Da_i\varphi\bigr]$ (the fractal Euler operator).

\begin{lemma}[Degree lemma]\label{lem:degree}
For every $\varphi\in\mathcal P_n^{(\alpha)}$, $E^{(\alpha)}[\varphi] = n\varphi.$
\end{lemma}
\begin{proof}
By linearity it suffices to check the identity on a monomial
$\varphi = c\,m_{p_0}(x_0)m_{p_1}(x_1)m_{p_2}(x_2)$ with $p_0+p_1+p_2=n$.
For each $i$, $D_i^{(\alpha)}$ lowers $p_i$ to $p_i-1$ with coefficient $1$ (rule (M)), and $X_i^{(\alpha)}$ raises it back to $p_i$ with coefficient $p_i$ (Definition~\ref{def:X-operator}; the result is $0$ when $p_i=0$), so $X_i^{(\alpha)}[D_i^{(\alpha)}\varphi]=p_i\varphi$. Summing over $i=0,1,2$ gives $(p_0+p_1+p_2)\varphi=n\varphi$.
\end{proof}

\begin{lemma}[{Commutator identity for $X_i^{(\alpha)}$}]\label{lem:leibniz-X}
For every $f\in\mathcal P^{(\alpha)}$ and $i\in\{0,1,2\}$,
$$
  \Delta^{(\alpha)}\bigl[X_i^{(\alpha)}f\bigr]
  = X_i^{(\alpha)}\bigl[\Delta^{(\alpha)}f\bigr]+2\,D_i^{(\alpha)}f.
$$
\end{lemma}
\begin{proof}
Write $\Delta^{(\alpha)}=\sum_{j=0}^2 (D_j^{(\alpha)})^2$. For $j\neq i$, $D_j^{(\alpha)}$ commutes with $X_i^{(\alpha)}$, so
\[
  (D_j^{(\alpha)})^2\bigl[X_i^{(\alpha)}f\bigr]
  = X_i^{(\alpha)}\bigl[(D_j^{(\alpha)})^2 f\bigr].
\]
For $j=i$, apply the exact Weyl relation (Lemma~\ref{lemma:X-weyl}) twice:
\[
  (D_i^{(\alpha)})^2\bigl[X_i^{(\alpha)}f\bigr]
  = D_i^{(\alpha)}\bigl[X_i^{(\alpha)}D_i^{(\alpha)}f + f\bigr] 
  = X_i^{(\alpha)}\bigl[(D_i^{(\alpha)})^2 f\bigr] + 2D_i^{(\alpha)}f.
\]
Summing over $j=0,1,2$ yields the claimed identity.
\end{proof}

\begin{corollary}\label{cor:leibniz-vec}
Let $\mathbf V = V_0\mathbf e_0 + V_1\mathbf e_1 + V_2\mathbf e_2 \in \mathcal A$ with $V_i\in\mathcal P^{(\alpha)}$. Then
\[
  \Delta^{(\alpha)}\Bigl[\sum_{i=0}^2 X_i^{(\alpha)}V_i\Bigr]
  = \sum_{i=0}^2 X_i^{(\alpha)}\bigl[\Delta^{(\alpha)}V_i\bigr]
    + 2\,\operatorname{div}^{(\alpha)}\mathbf V.
\]
\end{corollary}

On the polynomial space $\mathcal P^{(\alpha)}$ with monomial basis $\{m_p\}$, we define the \emph{Fischer inner product} by specifying it on the basis:
\[
	\langle m_p, m_q \rangle := \frac{\delta_{pq}}{p!},
\]
where $\delta_{pq}=1$ if $p=q$ and $0$ otherwise. This bilinear form is symmetric and positive definite, hence it defines a genuine inner product on each finite-dimensional subspace $\mathcal P_n^{(\alpha)}$.

Two linear operators $A,B:\mathcal P^{(\alpha)}\to\mathcal P^{(\alpha)}$ are said to be \emph{adjoint} with respect to this inner product if
\[
	\langle Af, g \rangle = \langle f, Bg \rangle \quad \text{for all } f,g\in\mathcal P^{(\alpha)}.
\]
Since the inner product is non-degenerate on each $\mathcal P_n^{(\alpha)}$, every operator $A$ admits a unique adjoint $A^\dagger$ satisfying $\langle Af,g\rangle=\langle f,A^\dagger g\rangle$.

\begin{lemma}[Properties of the Fischer inner product and $\dim H_n^{(\alpha)}$]\label{lem:fischer}
With the above notation and definitions, the following statements hold:
\begin{enumerate}[(i)]
  \item For all $f,g\in\mathcal P^{(\alpha)}$ and $i=0,1,2$, one has 
  $ 	\langle \Xa_i f, g \rangle = \langle f, \Da_i g \rangle.$ 
  Consequently, the adjoint of $\Lapa$ with respect to the Fischer inner product is given by 
  \[ 
  	(\Lapa)^\dagger = r^{2(\alpha)} = \sum_{i=0}^2 (\Xa_i)^2.
  \]

  \item The operator $r^{2(\alpha)}:\mathcal P^{(\alpha)}\to\mathcal P^{(\alpha)}$ is injective.

  \item For every $n\ge 2$, the map
  $ \Lapa:\mathcal P_n^{(\alpha)} \longrightarrow \mathcal P_{n-2}^{(\alpha)}$
  is surjective. As a consequence, $\dim H_n^{(\alpha)} = 2n+1$ for all $n\ge 0$.

  \item Let $H_n^{(\alpha),2}$ be the space of fractal polynomials of degree $n$ depending only on $(x_1,x_2)$ and satisfying $\big((\Da_1)^2+(\Da_2)^2\big) f = 0$. Then $\dim H_n^{(\alpha),2} = 2$ for all $n\ge 1$.
\end{enumerate}
\end{lemma}

\begin{proof}
(i) It suffices to verify the identity on basis monomials $f=m_p$ and $g=m_q$. 
By Definition~\ref{def:X-operator}, 
\[
\langle\Xa_i m_p,m_q\rangle=(p_i+1)\langle m_{p+\xi_i},m_q\rangle,
\]
which equals $(p_i+1)/(p+\xi_i)!=1/p!$ if $q=p+\xi_i$ and $0$ otherwise, since $(p+\xi_i)!=(p_i+1)\,p!$. 
On the other hand, $\langle m_p,\Da_i m_q\rangle=\langle m_p,m_{q-\xi_i}\rangle$ if $q_i\ge1$ and $0$ if $q_i=0$, 
which is again $1/p!$ if $q=p+\xi_i$ and $0$ otherwise. 
Therefore $\Da_i$ is the adjoint of $\Xa_i$; hence $((\Da_i)^2)^\dagger=(\Xa_i)^2$ and $(\Lapa)^\dagger=r^{2(\alpha)}$.

(ii) Let $f=\sum_p c_p m_p\neq0$ and let $p^*$ be the largest index with $c_{p^*}\neq0$. Since $r^{2(\alpha)}m_p=\sum_i(p_i+1)(p_i+2)\,m_{p+2\xi_i}$, the coefficient of $m_{p^*+2\xi_0}$ in $r^{2(\alpha)}f$ is $(p^*_0+1)(p^*_0+2)\,c_{p^*}\neq0$ from $p=p^*$. Thus $r^{2(\alpha)}f\neq0$.

(iii) The spaces $\mathcal P_n^{(\alpha)}$ are finite-dimensional and mutually orthogonal, and 
$\Lapa:\mathcal P_n^{(\alpha)}\to\mathcal P_{n-2}^{(\alpha)}$ is the adjoint of the injective map 
$r^{2(\alpha)}:\mathcal P_{n-2}^{(\alpha)}\to\mathcal P_n^{(\alpha)}$; hence it is surjective. 
Consequently, for $n\ge2$, 
\[
\dim H_n^{(\alpha)}=\dim\mathcal P_n^{(\alpha)}-\dim\mathcal P_{n-2}^{(\alpha)}
=\binom{n+2}{2}-\binom{n}{2}=2n+1,
\]
while for $n=0,1$ the Laplacian vanishes on $\mathcal P_n^{(\alpha)}$, whose dimensions are $1$ and $3$, respectively.

(iv) The arguments of (i)--(iii) apply verbatim to polynomials in the two variables $(x_1,x_2)$, 
with the Laplacian $(\Da_1)^2+(\Da_2)^2$ and $r^{2(\alpha)}$ replaced by $(\Xa_1)^2+(\Xa_2)^2$. 
For $n\ge2$ this gives $\dim H_n^{(\alpha),2}=(n+1)-(n-1)=2$; for $n=1$ all polynomials of degree one in $(x_1,x_2)$ 
are annihilated, and their space has dimension $2$.
\end{proof}

\begin{lemma}\label{lem:D0-onto}
For every $m\ge1$ the operator $\Da_0$ maps $H_m^{(\alpha)}$ onto $H_{m-1}^{(\alpha)}$. Consequently, $(\Da_0)^2$ maps $H_m^{(\alpha)}$ onto $H_{m-2}^{(\alpha)}$ for every $m\ge2$.
\end{lemma}
\begin{proof}
Since the operators $\Da_j$ commute and lower the degree by one, $\Da_0$ maps $H_m^{(\alpha)}$ into $H_{m-1}^{(\alpha)}$. An element of $H_m^{(\alpha)}$ annihilated by $\Da_0$ does not depend on $x_0$, because $\Da_0m_p=m_{p-\xi_0}$ vanishes exactly when $p_0=0$; hence $\ker\Da_0\cap H_m^{(\alpha)}=H_m^{(\alpha),2}$, which has dimension $2$ by Lemma~\ref{lem:fischer}(iv). The rank of $\Da_0$ on $H_m^{(\alpha)}$ is therefore $(2m+1)-2=2m-1=\dim H_{m-1}^{(\alpha)}$ by Lemma~\ref{lem:fischer}(iii).
\end{proof}

\begin{lemma}[Euler--Poincar\'e]\label{lem:euler-poincare}
Let $n\ge1$ and let $\mathbf g=g_0\mathbf e_0+g_1\mathbf e_1+g_2\mathbf e_2\in\mathcal A$ with $g_i\in\mathcal P_{n-1}^{(\alpha)}$ satisfy $\Da_ig_j=\Da_jg_i$ for all $i,j$ (equivalently, $\curla\mathbf g=\mathbf 0$). Then
\[
  \grada\,\Psi_3=\mathbf g,
  \qquad \Psi_3:=\frac1n\sum_{j=0}^2\Xa_jg_j\in\mathcal P_n^{(\alpha)}.
\]
\end{lemma}
\begin{proof}
By the commutation relations above and the assumption on $\mathbf g$,
\[
  \Da_i\sum_{j=0}^2\Xa_jg_j=g_i+\sum_{j=0}^2\Xa_j\Da_ig_j=g_i+\sum_{j=0}^2\Xa_j\Da_jg_i=g_i+E^{(\alpha)}[g_i]=n\,g_i ,
\]
where the last equality is Lemma~\ref{lem:degree} applied to $g_i\in\mathcal P_{n-1}^{(\alpha)}$.
\end{proof}

\subsection{Fractal-Homogeneous Solutions}
\label{subsec:Sn}

\begin{definition}[Fractal-homogeneous solutions]\label{def:Sn}
Let
\[
  S_n^{(\alpha)}
  := \bigl\{\mathbf u\in\mathcal P_{n,\mathcal A}^{(\alpha)} : L^{(\alpha)}[\mathbf u]=\mathbf 0\bigr\},
\]
the space of fractal-homogeneous solutions of \eqref{eq:lame_homogeneous} of combined degree $n$, where $L^{(\alpha)}$ is as in \eqref{eq:Lalpha-def}.
\end{definition}

\begin{lemma}[Dilatation is harmonic]\label{lem:divharm}
If $\mathbf u\in S_n^{(\alpha)}$ {with $n\ge1$}, then the fractal dilatation 
\[
e^{(\alpha)} := \operatorname{div}^{(\alpha)}\mathbf u \in H_{n-1}^{(\alpha)},
\]
i.e., it is fractal-harmonic of combined degree $n-1$.
\end{lemma}
\begin{proof}
Apply $\operatorname{div}^{(\alpha)}$ to the equation $L^{(\alpha)}[\mathbf u]=\mathbf 0$ in \eqref{eq:Lalpha-def}. Using the commutation $\operatorname{div}^{(\alpha)}\Delta^{(\alpha)}=\Delta^{(\alpha)}\operatorname{div}^{(\alpha)}$ (both sides are sums of commuting $D_j^{(\alpha)}D_j^{(\alpha)}$'s) and the identity $\operatorname{div}^{(\alpha)}\operatorname{grad}^{(\alpha)}=\Delta^{(\alpha)}$, we obtain
\[
  (\lambda^{(\alpha)}+2\mu^{(\alpha)})\,\Delta^{(\alpha)}e^{(\alpha)} = 0.
\]
By Postulate~\ref{post:hooke}(i), $\lambda^{(\alpha)}+2\mu^{(\alpha)}=\tfrac13(3\lambda^{(\alpha)}+2\mu^{(\alpha)})+\tfrac43\mu^{(\alpha)}>0$, hence $\Delta^{(\alpha)}e^{(\alpha)}=0$. Since each $D_i^{(\alpha)}$ lowers combined degree by one (rule (M)), we have $e^{(\alpha)}\in\mathcal P_{n-1}^{(\alpha)}$, and thus $e^{(\alpha)}\in H_{n-1}^{(\alpha)}$.
\end{proof}

\begin{corollary}\label{cor:dilatation}
Let $n\ge0$, $\mathbf u\in S_n^{(\alpha)}$ and $e^{(\alpha)}:=\diva\mathbf u$. Then $\Lapa\mathbf u=-\caa\,\grada e^{(\alpha)}$.
\end{corollary}
\begin{proof}
By \eqref{eq:Lalpha-def}, $L^{(\alpha)}[\mathbf u]=\mathbf 0$ reads $\mua\Lapa\mathbf u+(\laa+\mua)\grada e^{(\alpha)}=\mathbf 0$; divide by $\mua$.
\end{proof}

\begin{lemma}\label{lem:K}
For $\boldsymbol\varphi=\varphi_0\mathbf e_0+\varphi_1\mathbf e_1+\varphi_2\mathbf e_2\in\mathcal A$ with fractal-harmonic components in $\mathcal P^{(\alpha)}$, let
\[
  \mathcal K^{(\alpha)}[\boldsymbol\varphi]:=\kaa\,\boldsymbol\varphi-\grada\Bigl[\sum_{i=0}^2\Xa_i\varphi_i\Bigr].
\]
Then $\Lapa\,\mathcal K^{(\alpha)}[\boldsymbol\varphi]=-2\,\grada\diva\boldsymbol\varphi$, $\ \diva\,\mathcal K^{(\alpha)}[\boldsymbol\varphi]=(\kaa-2)\,\diva\boldsymbol\varphi$, and consequently
\[
  L^{(\alpha)}\bigl[\mathcal K^{(\alpha)}[\boldsymbol\varphi]\bigr] = 
  \Lapa\,\mathcal K^{(\alpha)}[\boldsymbol\varphi]+\caa\,\grada\diva\,\mathcal K^{(\alpha)}[\boldsymbol\varphi]=\mathbf 0.
\]
Moreover, $L^{(\alpha)}[\grada\varphi]=\mathbf 0$ for every fractal-harmonic $\varphi\in\mathcal P^{(\alpha)}$.
\end{lemma}
\begin{proof}
By Corollary~\ref{cor:leibniz-vec} and $\Lapa\varphi_i=0$, we have
\[
\Lapa\Bigl[\sum_{i=0}^2 \Xa_i \varphi_i\Bigr] = 2\,\diva \boldsymbol\varphi.
\]
Since $\Lapa\grada = \grada\Lapa$ and $\diva\grada = \Lapa$, it follows that
\[
\Lapa \mathcal K^{(\alpha)}[\boldsymbol\varphi]
= \kaa \Lapa \boldsymbol\varphi - \grada \Lapa\Bigl[\sum_{i=0}^2 \Xa_i \varphi_i\Bigr]
= -2\,\grada \diva \boldsymbol\varphi
\]
and
\[
\diva \mathcal K^{(\alpha)}[\boldsymbol\varphi]
= \kaa \diva \boldsymbol\varphi - \Lapa\Bigl[\sum_{i=0}^2 \Xa_i \varphi_i\Bigr]
= (\kaa - 2)\,\diva \boldsymbol\varphi.
\]
By \eqref{eq:constants}, $-2 + \caa(\kaa - 2) = 0$, which proves the first claim. {From this we conclude that} 
\begin{equation}
  \Lapa\grada\varphi = \grada\Lapa\varphi = \mathbf 0,
  \qquad
  \diva\grada\varphi = \Lapa\varphi = 0.
\end{equation}
\end{proof}

\subsection{A Fractal Kolosov--Muskhelishvili Formula}
\label{subsec:km}

\begin{definition}[The sharp map]\label{def:sharp}
\cite{Bock2014}. For $f=f_0+f_1\mathbf e_1+f_2\mathbf e_2+f_3\mathbf e_3\in\mathbb H$, define
\begin{equation}\label{eq:sharp}
  f^\sharp := f_0 - f_1\mathbf e_1 - f_2\mathbf e_2 + f_3\mathbf e_3.
\end{equation}
\end{definition}

A direct computation shows that $f\in\ker\overline\partial$ if and only if $f^\sharp\in\ker\partial$; thus the sharp map interchanges monogenic and anti-monogenic functions.

By \cite{Bock2014}, every displacement field $\mathbf u$ solving the classical homogeneous Lam\'e equation is the $\mathcal A$-part of an $\mathbb H$-valued field $\tilde{\mathbf u}$ of the form
\begin{equation}\label{eq:km-classical}
  2\mu\,\tilde{\mathbf u}
  := 4(1-\nu)\,\Phi 
     - \tfrac12\,\overline\partial\bigl[\bar x\,\Phi + \bar\Phi\,x\bigr] 
     - \Psi^\sharp,
\end{equation}
with $\Phi\in\ker\overline\partial\perp(\ker\overline\partial\cap\ker\partial)$ and $\Psi\in\ker\overline\partial$. This is the classical Kolosov--Muskhelishvili representation.

\begin{remark}[Position term of the classical formula]\label{rem:km-position}
Since $x\in\mathcal A$, $\bar x\,\Phi+\bar\Phi\,x=2\,\mathrm{Sc}(\bar x\,\Phi)=2\,(x_0\Phi_0+x_1\Phi_1+x_2\Phi_2)$, so that $\tfrac12\overline\partial[\bar x\,\Phi+\bar\Phi\,x]=\operatorname{grad}\bigl(\mathbf x\cdot\Phi_{\mathcal A}\bigr)$ with $\mathbf x\cdot\Phi_{\mathcal A}=\sum_{i=0}^2x_i\Phi_i$. Hence the $\mathcal A$-part of \eqref{eq:km-classical} reads $2\mu\mathbf u=4(1-\nu)\Phi_{\mathcal A}-\operatorname{grad}(\mathbf x\cdot\Phi_{\mathcal A})-\Psi^\sharp_{\mathcal A}$, and the fractal formula \eqref{eq:km-fractal} below is obtained by replacing $x_i\Phi_i$ with $X_i^{(\alpha)}[\Phi_i]$.
\end{remark}

In components, the multiplication table of $\mathbb H$ yields, for $f=f_0+f_1\mathbf e_1+f_2\mathbf e_2+f_3\mathbf e_3$,
\begin{equation}\label{eq:dbar-components}
\begin{aligned}
  \dbara f&=(\Da_0f_0-\Da_1f_1-\Da_2f_2)+(\Da_0f_1+\Da_1f_0+\Da_2f_3)\,\mathbf e_1\\
  &\quad+(\Da_0f_2+\Da_2f_0-\Da_1f_3)\,\mathbf e_2+(\Da_0f_3+\Da_1f_2-\Da_2f_1)\,\mathbf e_3,
\end{aligned}
\end{equation}
and $\dela f$ is obtained from \eqref{eq:dbar-components} by reversing the sign of every term that contains $\Da_1$ or $\Da_2$.

\begin{lemma}[Elementary properties of the sharp map]\label{lem:sharp-facts}
The following hold:
\begin{enumerate}
  \item If $\Psi=\sum_{i=0}^3 \Psi_i\,\mathbf e_i$ is fractal-monogenic, then the $\mathcal A$-part of its sharp, $\Psi^\sharp_{\mathcal A}:=\Psi_0 - \Psi_1\,\mathbf e_1 - \Psi_2\,\mathbf e_2$, satisfies $\operatorname{div}^{(\alpha)}\Psi^\sharp_{\mathcal A}=0$ exactly.
  \item For every quaternion $f$, $(f^\sharp)_{\mathcal A}=(f_{\mathcal A})^\sharp$ and $\bigl((f^\sharp)_{\mathcal A}\bigr)^\sharp=f_{\mathcal A}$. Moreover, $g$ is fractal-anti-monogenic ($\partial^{(\alpha)}[g]=0$) if and only if $g^\sharp$ is fractal-monogenic ($\overline\partial^{(\alpha)}[g^\sharp]=0$).
\end{enumerate}
\end{lemma}
\begin{proof}
(1) By \eqref{eq:dbar-components}, the scalar part of $\dbara\Psi=0$ is 
$\Da_0\Psi_0-\Da_1\Psi_1-\Da_2\Psi_2=0$, which is precisely $\diva\Psi^\sharp_{\mathcal A}=0$ 
because $\Psi^\sharp_{\mathcal A}=\Psi_0-\Psi_1\mathbf e_1-\Psi_2\mathbf e_2$.

(2) The first two identities follow directly from the definition of $\sharp$, which changes the signs 
of the $\mathbf e_1$- and $\mathbf e_2$-components only. For the last identity, it suffices to show 
$\dela[f^\sharp]=(\dbara f)^\sharp$ for every $f$. Set $g:=f^\sharp=f_0-f_1\mathbf e_1-f_2\mathbf e_2+f_3\mathbf e_3$. The components of $\dela g$ are
\[
\begin{aligned}
  &\Da_0f_0-\Da_1f_1-\Da_2f_2,\quad -(\Da_0f_1+\Da_1f_0+\Da_2f_3),\\
  &-(\Da_0f_2+\Da_2f_0-\Da_1f_3),\quad \Da_0f_3+\Da_1f_2-\Da_2f_1,
\end{aligned}
\]
which, by \eqref{eq:dbar-components}, are exactly the components of $(\dbara f)^\sharp$. 
Hence $\dbara f=0$ if and only if $\dela[f^\sharp]=0$; since $\sharp$ is an involution, 
$\dela g=0$ if and only if $\dbara[g^\sharp]=0$.
\end{proof}

\begin{theorem}[The fractal Kolosov--Muskhelishvili representation]\label{thm:X-km}
{Let $n\ge0$.}
\begin{enumerate}[(a)]
  \item {Let $\Phi,\Psi$ be fractal-monogenic with components in $\mathcal P_n^{(\alpha)}$, and define $\mathbf u$ by}
\begin{equation}\label{eq:km-fractal}
  2\mu^{(\alpha)}\mathbf u = 4(1-\nu^{(\alpha)})\,\Phi_{\mathcal A} - \operatorname{grad}^{(\alpha)}\Bigl[\sum_{i=0}^2 X_i^{(\alpha)}[\Phi_i]\Bigr] - \Psi_{\mathcal A}^\sharp
\end{equation}
{Then $\mathbf u\in S_n^{(\alpha)}$.}
  \item {Conversely, every $\mathbf u\in S_n^{(\alpha)}$ is of the form \eqref{eq:km-fractal} for some fractal-monogenic $\Phi,\Psi$ with components in $\mathcal P_n^{(\alpha)}$; moreover, $\Phi$ can be chosen as $\Phi=\dela[G]$ with $G\in H_{n+1}^{(\alpha)}$.}
\end{enumerate}
\end{theorem}

\begin{proof}
(a) Put $\boldsymbol\varphi:=\Phi_{\mathcal A}$ and $\mathbf w:=\Psi^\sharp_{\mathcal A}=\Psi_0-\Psi_1\mathbf e_1-\Psi_2\mathbf e_2$, so that $2\mua\mathbf u=\mathcal K^{(\alpha)}[\boldsymbol\varphi]-\mathbf w$, with $\mathcal K^{(\alpha)}$ as in Lemma~\ref{lem:K}. Since $\Lapa=\dbara\dela$ acts componentwise, every component of a fractal-monogenic field is fractal-harmonic; in particular, the components of $\boldsymbol\varphi$ and $\mathbf w$ are harmonic. By Lemma~\ref{lem:K}, $L^{(\alpha)}[\mathcal K^{(\alpha)}[\boldsymbol\varphi]]=\mathbf 0$. Moreover, $\diva\mathbf w=0$ by Lemma~\ref{lem:sharp-facts}(1), so that 
\[
	L^{(\alpha)}[\mathbf w] = \Lapa\mathbf w+\caa\grada\diva\mathbf w=\mathbf 0.
\]
By linearity, $L^{(\alpha)}[\mathbf u]=\mathbf 0$. The components of $\mathbf u$ lie in $\mathcal P_n^{(\alpha)}$, since $\Xa_i$ raises and $\Da_j$ lowers the degree by one; thus $\mathbf u\in S_n^{(\alpha)}$.

(b) Let $\mathbf u\in S_n^{(\alpha)}$ and $e^{(\alpha)}:=\diva\mathbf u$. If $n=0$, then $\mathbf u$ is constant and $e^{(\alpha)}=0$; we put $G:=0$, and all steps below remain valid.

\emph{Step 1.} For $n\ge1$, $e^{(\alpha)}\in H_{n-1}^{(\alpha)}$ by Lemma~\ref{lem:divharm}. Applying Lemma~\ref{lem:D0-onto} to $\Da_0:H_{n+1}^{(\alpha)}\to H_n^{(\alpha)}$ and to $\Da_0:H_n^{(\alpha)}\to H_{n-1}^{(\alpha)}$, we find $G\in H_{n+1}^{(\alpha)}$ with
\[
  (\Da_0)^2G=\frac{\mua}{\kaa-2}\,e^{(\alpha)} ,
\]
where $\kaa-2>0$ by \eqref{eq:constants}.

\emph{Step 2.} Put $\Phi:=\dela G=\Da_0G\,\mathbf e_0-\Da_1G\,\mathbf e_1-\Da_2G\,\mathbf e_2$. Since $\dbara\dela G=\Lapa G=0$, $\Phi$ is fractal-monogenic with components in $\mathcal P_n^{(\alpha)}$, and it is $\mathcal A$-valued. Using $\Lapa G=0$,
\[
  \diva\Phi=(\Da_0)^2G-(\Da_1)^2G-(\Da_2)^2G=2(\Da_0)^2G=\frac{2\mua}{\kaa-2}\,e^{(\alpha)} .
\]
Let $\mathbf u_1:=\frac1{2\mua}\mathcal K^{(\alpha)}[\Phi]$. By Lemma~\ref{lem:K}, $\diva\mathbf u_1=\frac{\kaa-2}{2\mua}\diva\Phi=e^{(\alpha)}$ and 
\[
	\Lapa\mathbf u_1=-\frac1{\mua}\grada\diva\Phi=-\frac{2}{\kaa-2}\grada e^{(\alpha)}=-\caa\,\grada e^{(\alpha)}.
\]
By Corollary~\ref{cor:dilatation}, the field $\mathbf r:=\mathbf u-\mathbf u_1$ satisfies $\diva\mathbf r=0$ and 
\[
	\Lapa\mathbf r=-\caa\grada e^{(\alpha)}+\caa\grada e^{(\alpha)}=\mathbf 0.
\]

\emph{Step 3.} Put $\mathbf w:=-2\mua\mathbf r=w_0\mathbf e_0+w_1\mathbf e_1+w_2\mathbf e_2$; its components are harmonic, belong to $\mathcal P_n^{(\alpha)}$, and $\diva\mathbf w=0$. Let $\mathbf g:=\curla\mathbf w$. By \eqref{eq:curlcurl}, 
\[
	\curla\mathbf g=\grada\diva\mathbf w-\Lapa\mathbf w=\mathbf 0,
\]
that is, $\Da_ig_j=\Da_jg_i$ for $i,j=0,1,2$. For $n\ge1$, Lemma~\ref{lem:euler-poincare} yields $\Psi_3:=\frac1n\sum_{j=0}^2\Xa_jg_j\in\mathcal P_n^{(\alpha)}$ with $\grada\Psi_3=\mathbf g$; for $n=0$ we have $\mathbf g=\mathbf 0$ and put $\Psi_3:=0$.

\emph{Step 4.} Define $\Psi:=w_0-w_1\mathbf e_1-w_2\mathbf e_2+\Psi_3\mathbf e_3$. By \eqref{eq:dbar-components}, $\dbara\Psi=0$ is equivalent to the four equations
\[
\begin{aligned}
  &\Da_0w_0+\Da_1w_1+\Da_2w_2=0, &&\Da_2\Psi_3=\Da_0w_1-\Da_1w_0,\\
  &\Da_1\Psi_3=\Da_2w_0-\Da_0w_2, &&\Da_0\Psi_3=\Da_1w_2-\Da_2w_1,
\end{aligned}
\]
of which the first is $\diva\mathbf w=0$ and the last three are $\grada\Psi_3=\mathbf g$. Hence $\Psi$ is fractal-monogenic with components in $\mathcal P_n^{(\alpha)}$, and $\Psi^\sharp_{\mathcal A}=\mathbf w$. Finally,
\[
  2\mua\mathbf u=2\mua\mathbf u_1+2\mua\mathbf r=\mathcal K^{(\alpha)}[\Phi]-\mathbf w ,
\]
which is \eqref{eq:km-fractal} with $\Phi_{\mathcal A}=\Phi$.
\end{proof}

\begin{corollary}[Polynomial solutions of arbitrary degree]\label{cor:mixed}
A field $\mathbf u\in\mathcal P_{\mathcal A}^{(\alpha)}$ satisfies $L^{(\alpha)}[\mathbf u]=\mathbf 0$ if and only if there exist fractal-monogenic fields $\Phi,\Psi$ with components in $\mathcal P^{(\alpha)}$, of arbitrary degrees, such that \eqref{eq:km-fractal} holds. Moreover, $\Phi$ can be chosen as $\Phi=\dela[G]$ with a fractal-harmonic $G\in\mathcal P^{(\alpha)}$.
\end{corollary}
\begin{proof}
Every $f\in\mathcal P^{(\alpha)}$ is a unique finite sum of homogeneous components $f_n\in\mathcal P_n^{(\alpha)}$. The operators $\Da_i$ lower and $\Xa_i$ raise the degree by one; hence $\dbara$, $\dela$ and $L^{(\alpha)}$ map homogeneous components to homogeneous components, with degree shifts $-1$, $-1$ and $-2$, respectively, and the map $(\Phi,\Psi)\mapsto\mathbf u$ in \eqref{eq:km-fractal}, in which $\Xa_i$ and $\grada$ compensate each other, preserves the degree.

Let $\Phi,\Psi$ be fractal-monogenic. Then $\dbara\Phi=\sum_n\dbara\Phi_n$ with $\dbara\Phi_n\in\mathcal P_{n-1}^{(\alpha)}\otimes\mathbb H$, so that all $\dbara\Phi_n$ vanish, and likewise for $\Psi$. By Theorem~\ref{thm:X-km}(a) the field $\mathbf u_n$ defined by \eqref{eq:km-fractal} from $(\Phi_n,\Psi_n)$ lies in $S_n^{(\alpha)}$, and $\mathbf u=\sum_n\mathbf u_n$ solves $L^{(\alpha)}[\mathbf u]=\mathbf 0$. 

Conversely, let $\mathbf u=\sum_n\mathbf u_n$ with $L^{(\alpha)}[\mathbf u]=\sum_nL^{(\alpha)}[\mathbf u_n]=\mathbf 0$. As $L^{(\alpha)}[\mathbf u_n]\in\mathcal P_{n-2,\mathcal A}^{(\alpha)}$, each $L^{(\alpha)}[\mathbf u_n]$ vanishes, i.e.\ $\mathbf u_n\in S_n^{(\alpha)}$. By Theorem~\ref{thm:X-km}(b), $\mathbf u_n$ is given by \eqref{eq:km-fractal} with $\Phi_n=\dela[G_n]$, $G_n\in H_{n+1}^{(\alpha)}$, and fractal-monogenic $\Psi_n$. Summing over the finitely many $n$ with $\mathbf u_n\neq\mathbf 0$, and using the linearity of \eqref{eq:km-fractal} in $(\Phi,\Psi)$, we obtain the claim with $\Phi=\sum_n\Phi_n$, $\Psi=\sum_n\Psi_n$ and $G=\sum_nG_n$.
\end{proof}

\begin{corollary}[Fractal Papkovich--Neuber representation]\label{cor:PN}
Let $n\ge0$.
\begin{enumerate}[(a)]
  \item If $\varphi\in H_{n+1}^{(\alpha)}$ and $\boldsymbol\psi\in H_{n,\mathcal A}^{(\alpha)}$, then
  \begin{equation}\label{eq:PN-fractal}
    2\mua\mathbf u=\grada\Bigl[\varphi+\sum_{i=0}^2\Xa_i\psi_i\Bigr]-\kaa\boldsymbol\psi
  \end{equation}
  belongs to $S_n^{(\alpha)}$.
  \item Every $\mathbf u\in S_n^{(\alpha)}$ is of the form \eqref{eq:PN-fractal}.
\end{enumerate}
\end{corollary}
\begin{proof}
(a) The right-hand side of \eqref{eq:PN-fractal} equals $\grada\varphi-\mathcal K^{(\alpha)}[\boldsymbol\psi]$, and both terms are annihilated by $L^{(\alpha)}$ by Lemma~\ref{lem:K}. Its components lie in $\mathcal P_n^{(\alpha)}$.

(b) Let $\Phi,\Psi$ be as in Theorem~\ref{thm:X-km}(b), $\mathbf w:=\Psi^\sharp_{\mathcal A}$, and set
\[
  \boldsymbol\psi:=-\Phi+\tfrac1{\kaa}\mathbf w,\qquad \varphi:=-\tfrac1{\kaa}\sum_{i=0}^2\Xa_iw_i .
\]
The components of $\boldsymbol\psi$ are harmonic. By Corollary~\ref{cor:leibniz-vec} and $\diva\mathbf w=0$ (Lemma~\ref{lem:sharp-facts}(1)), $\Lapa\varphi=-\tfrac2{\kaa}\diva\mathbf w=0$; also $\varphi\in\mathcal P_{n+1}^{(\alpha)}$. Since $\varphi+\sum_{i=0}^2\Xa_i\psi_i=-\sum_{i=0}^2\Xa_i\Phi_i$, we obtain 
\[
	\grada\Bigl[\varphi+\sum_{i=0}^2\Xa_i\psi_i\Bigr]-\kaa\boldsymbol\psi=-\grada\Bigl[\sum_{i=0}^2\Xa_i\Phi_i\Bigr]+\kaa\Phi-\mathbf w=2\mua\mathbf u.
\]
\end{proof}

\begin{remark}[Normalization and redundancy]\label{rem:e3}
Before projection to $\mathcal A$, the right-hand side of \eqref{eq:km-fractal} has the $\mathbf e_3$-component $\kaa\Phi_3-\Psi_3$, since the position term is $\mathcal A$-valued. In \cite{Bock2014} the potentials are normalized so that this component vanishes and $\Phi$ is chosen orthogonal to the monogenic constants, which removes part of the non-uniqueness of the representation; Theorem~\ref{thm:X-km} does not impose these conditions.
\end{remark}

\begin{remark}[Galerkin representation]
Let $\mathbf X^{(\alpha)}:=(\Xa_0,\Xa_1,\Xa_2)$ and 
\[
	T_G^{(\alpha)}[\mathbf G]:=2(1-\nua)\Lapa\mathbf G-\grada\diva\mathbf G
\] 
for $\mathbf G\in\mathcal P_{\mathcal A}^{(\alpha)}$. A direct computation gives $\Lapa T_G^{(\alpha)}[\mathbf G]=2(1-\nua)(\Lapa)^2\mathbf G-\grada\diva\Lapa\mathbf G$ and $\diva T_G^{(\alpha)}[\mathbf G]=(1-2\nua)\diva\Lapa\mathbf G$, so that, by $\caa(1-2\nua)=1$, 
\[
	L^{(\alpha)}[T_G^{(\alpha)}\mathbf G]=2\mua(1-\nua)(\Lapa)^2\mathbf G.
\] 
Moreover, $T_G^{(\alpha)}:\mathcal P_{n+2,\mathcal A}^{(\alpha)}\to\mathcal P_{n,\mathcal A}^{(\alpha)}$ is onto: with respect to the Fischer inner product of Lemma~\ref{lem:fischer} its adjoint is $\mathbf w\mapsto2(1-\nua)r^{2(\alpha)}\mathbf w-\mathbf X^{(\alpha)}\bigl(\mathbf X^{(\alpha)}\cdot\mathbf w\bigr)$, and if this vanishes then applying $\mathbf X^{(\alpha)}\cdot$ gives $(1-2\nua)\,r^{2(\alpha)}\bigl(\mathbf X^{(\alpha)}\cdot\mathbf w\bigr)=0$, whence $\mathbf X^{(\alpha)}\cdot\mathbf w=0$ by Lemma~\ref{lem:fischer}(ii) and then $\mathbf w=\mathbf 0$. Consequently every $\mathbf u\in S_n^{(\alpha)}$ is of the form $T_G^{(\alpha)}[\mathbf G]$ with $(\Lapa)^2\mathbf G=\mathbf 0$, which is the fractal counterpart of the Boussinesq--Galerkin representation \cite{Galerkin1930,Mindlin1936}. The argument uses $\nua\neq\tfrac12$.
\end{remark}

\begin{remark}[Relation to the classical polynomial calculus]\label{rem:iso}
Let $\mathcal P$ denote the real polynomials in $(x_0,x_1,x_2)$ and let 
$\mathcal T:\mathcal P\to\mathcal P^{(\alpha)}$ be the linear bijection defined by 
\[
\mathcal T\Bigl[\dfrac{x_0^{p_0}}{p_0!} \dfrac{x_1^{p_1}}{p_1!} \dfrac{x_2^{p_2}}{p_2!}\Bigr] = m_{p_0}(x_0)m_{p_1}(x_1)m_{p_2}(x_2).
\]
Comparing the actions on basis monomials with rule~\eqref{rule:monomial} and Definition~\ref{def:X-operator}, one obtains 
\[
	\mathcal T\circ\partial_{x_i}=\Da_i\circ\mathcal T \qquad\text{and}\qquad \mathcal T\circ(x_i\,\cdot)=\Xa_i\circ\mathcal T.
\] 
Moreover $\mathcal T$ preserves the grading and commutes with real scalars and quaternionic units. Consequently, every identity between polynomial fields that is built from real linear combinations of compositions of the operators $\partial_{x_i}$ and multiplications by $x_i$ (in particular the operators $\Delta$, $\operatorname{grad}$, $\operatorname{div}$, $\overline\partial$, $\partial$, the Lam\'e operator, and the sets of harmonic and monogenic polynomials) has a counterpart in $\mathcal P^{(\alpha)}$ with the same coefficients. Theorem~\ref{thm:X-km} is thus the image under $\mathcal T$ of its classical case $\alpha=1$, which is a statement on homogeneous polynomial solutions of the classical Lam\'e system; its coefficients depend on $\alpha$ only through $\nua$. 

The map $\mathcal T$ is linear but not multiplicative: for functions of the same variable, the product in $\mathcal P^{(\alpha)}$ does not correspond to the pointwise product in $\mathcal P$ when $\alpha\neq1$. This is why multiplication by $x_i$ in the classical formulae corresponds to $\Xa_i$ and not to a pointwise product (Remark~\ref{rem:X-grading}). The proofs in this section use only the Weyl relation and Fischer duality and do not rely on $\mathcal T$.
\end{remark}

\section{An Example}
\label{sec:examples}

We illustrate Theorem~\ref{thm:X-km} {in the case} $n=2$ and $\Psi=0$. Let
\[
  G:=m_3(x_0)-m_1(x_0)\,m_2(x_1)\in H_3^{(\alpha)}.
\]
Indeed, $(\Da_0)^2G=m_1(x_0)$, $(\Da_1)^2G=-m_1(x_0)$ and $(\Da_2)^2G=0$, so that $\Lapa G=0$. The field
\[
  \Phi:=\dela G=\bigl[m_2(x_0)-m_2(x_1)\bigr]\mathbf e_0+m_1(x_0)\,m_1(x_1)\,\mathbf e_1
\]
is $\mathcal A$-valued and fractal-monogenic, because $\dbara\dela G=\Lapa G=0$. 

Using $\Xa_i[m_p(x_i)g]=(p+1)m_{p+1}(x_i)g$, {we have}
\[
  \sum_{i=0}^2\Xa_i\Phi_i=3\,m_3(x_0)-m_1(x_0)m_2(x_1)+2\,m_1(x_0)m_2(x_1)=3\,m_3(x_0)+m_1(x_0)\,m_2(x_1).
\]
{From \eqref{eq:km-fractal} it follows that}
\[
  2\mua\mathbf u=\bigl[(\kaa-3)\,m_2(x_0)-(\kaa+1)\,m_2(x_1)\bigr]\mathbf e_0+(\kaa-1)\,m_1(x_0)\,m_1(x_1)\,\mathbf e_1 .
\]
{Theorem~\ref{thm:X-km}(a) now gives $\mathbf u\in S_2^{(\alpha)}$, which can also viewed directly from} rule~\eqref{rule:monomial}: $\Lapa(2\mua\mathbf u)=\bigl[(\kaa-3)-(\kaa+1)\bigr]\mathbf e_0=-4\,\mathbf e_0$, while $2\mua\diva\mathbf u=2(\kaa-2)\,m_1(x_0)$ and hence $2\mua\grada\diva\mathbf u=2(\kaa-2)\,\mathbf e_0.$ Thus $\Lapa\mathbf u+\caa\grada\diva\mathbf u=\frac1{2\mua}\bigl(-4+2\caa(\kaa-2)\bigr)\mathbf e_0=\mathbf 0$ by \eqref{eq:constants}, for every $\alpha\in(0,1]$. For $\alpha=1$ {we arrive to} a classical quadratic solution of the Lam\'e system.

\section{Conclusions}
\label{sec:concl}

{We presented} a quaternionic formulation of linear elasticity based on Yang's local fractional derivative. {In addition,} we heralded suitable fractal kinematics, a constitutive law incorporating a material length scale, and the equilibrium equation lead to the fractal Lam\'e system, which is shown to be equivalent to the monogenicity of an explicit quaternionic field (Theorem~\ref{thm:Q}). 
On the polynomial sector, the classical Kolosov--Muskhelishvili representation admits a fractal counterpart (Theorem~\ref{thm:X-km} and Corollary~\ref{cor:mixed}): every pair of fractal-monogenic potentials generates a polynomial solution of the homogeneous Lam\'e system, and every such solution arises in this way. The proof is direct and relies on a corrected position operator $X_i^{(\alpha)}$, which satisfies the exact Weyl relation with $D_i^{(\alpha)}$ and replaces the position vector in the classical formula. 
The fractal Papkovich--Neuber representation is obtained as a corollary (Corollary~\ref{cor:PN}), and the Galerkin representation is discussed in a remark.

All of the above concerns the polynomial sector $\mathcal P^{(\alpha)}$. 
On this sector, the triple $(\mathcal P^{(\alpha)},D_i^{(\alpha)},X_i^{(\alpha)})$ is isomorphic 
to the classical polynomial calculus (Remark~\ref{rem:iso}); consequently, the coefficients in the representations are independent of $\alpha$, and the results constitute fractal counterparts of classical results on polynomial solutions of the Lam\'e system \cite{Bock2014,WeiszPatraultBockGurlebeck2014}. 
This isomorphism is linear but not multiplicative, which is precisely why the position operator $X_i^{(\alpha)}$, rather than pointwise multiplication by $x_i$, is required. The fractal parameter enters through the constitutive parameters $\lambda^{(\alpha)}$, $\mu^{(\alpha)}$, the length scale $\ell$, and the interpretation of the basis monomials $m_p$ as functions on Cantor-type sets.

Several questions remain open. Within the polynomial sector, these include explicit stress formulae, a normalization of the potentials that removes the redundancy in Theorem~\ref{thm:X-km} (as in the classical case \cite{Bock2014}), and a $\psi$-hyperholomorphic refinement that avoids 
this redundancy altogether. 
Beyond the polynomial sector, a representation theory for series built from $E_\alpha$, $\sin_\alpha$, and $\cos_\alpha$ would require a separate convergence analysis, since the normalizing factors $\Gamma(1+k\alpha)$ grow more slowly than $k!$ when $\alpha<1$. A genuine boundary-value theory 
would further require the specification of a domain, an appropriate function space, and local fractional integral representations of Borel--Pompeiu type on Cantor-type sets. 
Because the isomorphism with the classical polynomial calculus is not multiplicative, it does not extend automatically to variable coefficients or nonlinear constitutive laws, nor does it by itself relate $\lambda^{(\alpha)}$, $\mu^{(\alpha)}$, and $\ell$ to the microstructure of fractal media. {Both issues are the subject of ongoing work}.

\bigskip
\section*{Statements and Declarations}

\subsection*{Funding} 
This research is funded by University of Transport and Communications (UTC) under grant number T2027-CNTT-xxx.

\subsection*{Conflict of Interest} 
The authors declare that they have no known competing financial interests or personal relationships that could have appeared to influence the work reported in this paper.

\subsection*{Author Contributions} 
HMN conceived the study and wrote the paper. JBR reviewed the project. All authors provided critical feedback and helped shape the research, analysis, and manuscript.

\subsection*{ORCID} 
\noindent
Hung Manh Nguyen: \url{https://orcid.org/0009-0004-6503-6693}\\
Juan Bory-Reyes: \url{https://orcid.org/0000-0002-7004-1794}



\bigskip
\begin{thebibliography}{9}

\bibitem{AlvarezPena2013}
C.~\'Alvarez-Pe\~na,
\textit{Contragenic Functions and Appell Bases for Monogenic Functions
of Three Variables}, Ph.D.\ thesis, CINVESTAV del I.P.N., Mexico, 2013.

\bibitem{AlvarezPenaPorter2014}
C.~\'Alvarez-Pe\~na and R.~M.~Porter,
\textit{Contragenic functions of three variables}, Complex Anal.\
Oper.\ Theory \textbf{8} (2014), 409--427.

\bibitem{Balankin2015}
A.~S.~Balankin,
\textit{A continuum framework for mechanics of fractal materials I:
from fractional space to continuum with fractal metric},
Eur.\ Phys.\ J.\ B \textbf{88} (2015), 90.

\bibitem{BalankinBoryShapiro2016}
A. ~S. ~Balankin, J. ~Bory-Reyes and M. ~Shapiro,
\textit{Towards a physics on fractals: Differential vector calculus in
three-dimensional continuum with fractal metric}, Phys. A \textbf{444} (2016), 345-359.

\bibitem{Bock2014}
S.~Bock,
\textit{On monogenic series expansions with applications to linear
elasticity}, Adv.\ Appl.\ Clifford Algebr.\ \textbf{24} (2014), 931--943.

\bibitem{BockGurlebeck2009}
S.~Bock and K.~G\"urlebeck,
\textit{On a spatial generalization of the Kolosov--Muskhelishvili
formulae}, Math.\ Meth.\ Appl.\ Sci.\ \textbf{32} (2009), 223--240.

\bibitem{BockLegatiukNguyen2015}
S.~Bock, K.~G\"urlebeck, D.~Legatiuk and H.~M.~Nguyen,
\textit{$\psi$-hyperholomorphic functions and a Kolosov--Muskhelishvili
formula}, Math.\ Meth.\ Appl.\ Sci.\ \textbf{38} (2015), 5114--5123.

\bibitem{BoryNIDS2026}
{J.~Bory-Reyes, M.~A.~P\'erez-de la Rosa, J.~O.~Gonz\'alez-Cervantes and J.~E.~Napoles-Vald\'es,
\textit{A non-integer dimensional space approach to the Moisil--Teodorescu operator} (2026),
arXiv:2606.30350.}

\bibitem{BoryPerezPena2022}
{J.~Bory-Reyes, M.~A.~P\'erez-de la Rosa and Y.~Pe\~na-P\'erez,
\textit{Fractional Moisil--Teodorescu operator in elasticity and electromagnetism},
Alexandria Eng.\ J.\ \textbf{61} (2022), 6811--6818.}

\bibitem{BoryRosas2021}
J.~Bory-Reyes and M.~A.~P\'erez-de la Rosa,
\textit{Local fractional Moisil--Teodorescu operator in quaternionic
setting involving Cantor-type coordinate systems},
Math.\ Meth.\ Appl.\ Sci.\ \textbf{44} (2021), 605--616.

\bibitem{CarpinteriCornetti2002}
A.~Carpinteri and P.~Cornetti,
\textit{A fractional calculus approach to the description of stress
and strain localization in fractal media},
Chaos, Solitons \& Fractals \textbf{13} (2002), 85--94.

\bibitem{Galerkin1930}
B.~G.~Galerkin,
\textit{Contribution \`a la solution g\'en\'erale du probl\`eme de la
th\'eorie de l'\'elasticit\'e dans le cas de trois dimensions},
Comptes Rendus Acad.\ Sci.\ Paris \textbf{190} (1930), 1047--1048.

\bibitem{GurlebeckNguyen2013}
K.~G\"urlebeck and H.~M.~Nguyen,
\textit{On $\psi$-hyperholomorphic functions in $\mathbb R^3$},
AIP Conf.\ Proc.\ \textbf{1558} (2013), 496--501.

\bibitem{GurlebeckNguyen2014}
K.~G\"urlebeck and H.~M.~Nguyen,
\textit{On $\psi$-hyperholomorphic functions and a decomposition of
harmonics}, in: \textit{Trends in Mathematics: Hypercomplex Analysis:
New Perspectives and Applications}, Birkh\"auser, Basel, 2014, 181--189.

\bibitem{GurlebeckNguyenLegatiuk2014}
K.~G\"urlebeck, H.~M.~Nguyen and D.~Legatiuk,
\textit{An additive decomposition of harmonic functions in $\mathbb
R^3$}, in: Computational Science and Its Applications -- ICCSA 2014,
Lecture Notes in Comput.\ Sci.\ \textbf{8579}, Springer, 2014, 189--203.

\bibitem{KleinObbink1993}
B.~Klein Obbink,
\textit{On the solutions of $D^nD^mF$}, Reports on Applied and
Numerical Analysis, Eindhoven University of Technology, Department of
Mathematics and Computing Science, 1993.

\bibitem{Lazopoulos2019}
K.~A.~Lazopoulos and A.~K.~Lazopoulos,
\textit{Fractional Differential Calculus and Continuum Mechanics},
in G.~Voyiadjis (ed.),
\textit{Handbook of Nonlocal Continuum Mechanics for Materials and Structures},
Springer, Cham, 2019.

\bibitem{Legatiuk2021}
D.~Legatiuk and D.~Weisz-Patrault,
\textit{Coupling of complex function theory and finite element method for crack propagation through energetic formulation: conformal mapping approach and reduction to a Riemann-Hilbert problem},
Comput. Methods Funct. Theory. 22, 535–557 (2022).

\bibitem{Mindlin1936}
R.~D.~Mindlin,
\textit{Note on the Galerkin and Papkovitch stress functions},
Bull. Amer. Math. Soc. \textbf{42} (1936), no. 6, 373--376.

\bibitem{Muskhelishvili1977}
N.~I.~Muskhelishvili,
\textit{Some Basic Problems of the Mathematical Theory of Elasticity},
Springer, Dordrecht, 1977.

\bibitem{PatnaikSidhardhSemperlotti2020}
S.~Patnaik, S.~Sidhardh, and F.~Semperlotti,
\textit{Fractional-order models for the static and dynamic analysis of nonlocal plates},
Commun. Nonlinear Sci. Numer. Simul. \textbf{95} (2021), Paper No. 105601, 25 pp.

\bibitem{PatnaikSemperlotti2020}
S.~Patnaik and F.~Semperlotti,
\textit{A generalized fractional-order elastodynamic theory for non-local attenuating media},
Proc.\ R.\ Soc.\ A \textbf{476} (2020), 20200200.

\bibitem{RahmatBaleanuYang2015}
M.~R.~S.~Rahmat, D.~Baleanu and X.~J.~Yang,
\textit{Cantor-type spherical-coordinate method for differential equations
within local fractional derivatives},
in C.~Cattani, H.~M.~Srivastava and X.~J.~Yang (eds.),
\textit{Fractional Dynamics},
De Gruyter Open, Warsaw, 2015, 231--242.

\bibitem{SumelkaBlaszczyk2014}
W.~Sumelka and T.~B{\l}aszczyk,
\textit{Fractional continua for linear elasticity},
Arch.\ Mech.\ \textbf{66} (2014), 147--172.

\bibitem{VecCalc2025}
J.~J.~Gohil, K.~B.~Kachhia, and K.~M.~Kolwankar,
\textit{On local fractional vector calculus for fields in fractal space},
Physics Letters A \textbf{556} (2025), 130830.

\bibitem{WeiszPatraultBockGurlebeck2014}
D.~Weisz-Patrault, S.~Bock and K.~G\"urlebeck,
\textit{Three-dimensional elasticity based on quaternion-valued
potentials}, Int.\ J.\ Solids Struct.\ \textbf{51} (2014), 3422--3430.

\bibitem{Yang2015}
X.~J.~Yang, D.~Baleanu and H.~M.~Srivastava,
\textit{Local Fractional Integral Transforms and Their Applications},
Academic Press, Amsterdam, 2015.

\bibitem{Yang2012book}
X.~J.~Yang,
\textit{Advanced Local Fractional Calculus and Its Applications},
World Science Publisher, New York, 2012.

\end{thebibliography}
\end{document}